\documentclass[10pt]{article}
\usepackage[margin=1in]{geometry}
\usepackage{graphicx,xcolor,cite,mathtools,bbold}
\usepackage{amssymb,amsmath,amsthm,mathrsfs,paralist,bm,esint,setspace}
\RequirePackage[colorlinks,citecolor=blue,urlcolor=blue]{hyperref}
\usepackage{cases}

\usepackage{color}
\usepackage{setspace}
\let\OLDthebibliography\thebibliography
\renewcommand\thebibliography[1]{
  \OLDthebibliography{#1}
  \setlength{\parskip}{3pt}
  \setlength{\itemsep}{0pt plus 0.3ex}
}

\DeclareMathOperator{\Var}{Var}

\newcommand{\<}{\langle}
\renewcommand{\>}{\rangle}

\newcommand{\N}{\mathbb{N}}
\newcommand{\Z}{\mathbb{Z}}

\newcommand{\R}{\mathbb{R}}

\newcommand{\F}{\mathcal{F}}
\newcommand{\T}{\mathbb{T}}

\newcommand{\lip}{\text{\rm Lip}}

\renewcommand{\P}{\mathrm{P}}
\newcommand{\E}{\mathrm{E}}
\newcommand{\bcA}{\bm{\mathcal{A}}}

\newcommand{\1}{\mathbb{1}}
\renewcommand{\d}{{\rm d}}

\newcommand{\e}{{\rm e}}
\renewcommand{\geq}{\geqslant}
\renewcommand{\leq}{\leqslant}
\renewcommand{\ge}{\geqslant}
\renewcommand{\le}{\leqslant}

\author{Davar Khoshnevisan\\University of Utah
	\and Kunwoo Kim\\POSTECH
	\and Carl Mueller\\University of Rochester}
\title{\bf Dissipation in Parabolic SPDEs II: \\Oscillation and decay of the solution\thanks{
	Research supported in part by the NSF grant DMS-1855439 [D.K.], 
NRF grants  2019R1A5A1028324 and 2020R1A2C4002077 [K.K], and Simons Foundation grant 513424 [C.M.].  
}}
\date{January 24, 2022}
\newtheorem{stat}{Statement}[section]
\newtheorem{proposition}[stat]{Proposition}

\newtheorem{corollary}[stat]{Corollary}
\newtheorem{theorem}[stat]{Theorem}
\newtheorem{lemma}[stat]{Lemma}
\theoremstyle{definition}

\newtheorem{remark}[stat]{Remark}
\newtheorem*{OP}{Open Problem}

\numberwithin{equation}{section}

\begin{document}
\maketitle
\begin{abstract} 
	We consider a stochastic heat equation of the type,
	$\partial_t u = \partial^2_x u + \sigma(u)\dot{W}$ on $(0\,,\infty)\times[-1\,,1]$
	with periodic boundary conditions and non-degenerate positive initial data, where
	$\sigma:\R\to\R$ is a non-random Lipschitz continuous function
	and $\dot{W}$ denotes space-time white noise. If additionally $\sigma(0)=0$ then the solution is
	known to be strictly positive; see Mueller \cite{Mueller1}. In that case, we prove that the oscillation of the 
	logarithm of the solution decays sublinearly as time tends to infinity. 
	Among other things, it follows that, with probability one, all limit points of 
	$t^{-1}\sup_{x\in[-1,1]} \log u(t\,,x)$ and $t^{-1}\inf_{x\in[-1,1]} \log u(t\,,x)$ must coincide. 
	As a consequence of this fact, we prove that, when
	$\sigma$ is linear,  there is a.s.\ only one such
	limit point and hence the entire path decays almost surely
	at an exponential rate.\\[.1in]

\end{abstract}

\tableofcontents

\section{Introduction}

Let $\T=\R/(2\Z)$ denote the 1-dimensional torus, and identify $\T$ with the interval
$[-1\,,1]$ in the usual way. We are interested in the large-time behavior
of the unique continuous solution $u$ to the following stochastic heat equation,
\begin{equation}\label{SHE}
	\partial_t u(t\,,x) = \partial^2_x u(t\,,x) + \sigma(u(t\,,x))\dot{W}(t\,,x)
	\qquad\text{for all $(t\,,x)\in(0\,,\infty)\times\T$},
\end{equation}
where $\dot{W}$ denotes a space-time white noise and $\sigma:\R\to\R$
is a non-random, Lipschitz continuous function that satisfies the following:
\begin{align}\label{sigma(0)=0}
	\sigma(0) = 0
	\quad\text{and}\quad
	|\sigma(z_1)-\sigma(z_2)|\le \lip(\sigma)|z_1-z_2|
	\quad\text{for all $z_1,z_2\in\R$.}
\end{align}
We consider \eqref{SHE} subject to having the initial
profile $u_0:\T\to\R_+$ which is assumed to be an element of $L^\infty(\T)$;
that is, $u_0$ a non-negative, bounded and measurable function that is either non-random or
random but independent of the noise $\dot{W}$. 
Let us mention also that 
the quotient topology of $\T$ automatically imposes a periodic 
boundary condition on \eqref{SHE}; that is, \eqref{SHE} is tacitly 
restricted additionally to satisfy $u(t\,,-1)=u(t\,,1)$ for all $t>0$.

With these assumptions in place, standard arguments show that \eqref{SHE} has a unique 
random-field solution valid for all times $t>0$; see Walsh \cite{Walsh}, Chapter 3.
Walsh's presentation is for the same SPDE but with different boundary conditions.
Small adjustments to that argument will establish the existence and uniqueness of a solution
in the present setting.  

We pause to remind that \eqref{SHE} does 
not make sense if we interpret it literally as written, since we do not expect $u$ to be 
differentiable in either of its variable.  As was pointed out earlier, for example by Walsh \cite{Walsh}, 
\eqref{SHE} is shorthand for the random evolution equation \eqref{mild} below. 
The latter
is also sometimes known as the ``mild,'' or ``integral,'' formulation of \eqref{SHE}.

To avoid degeneracies, we will also assume
that $\inf_{x\in\T} u_0(x)>0$.%
\footnote{The condition $\inf_\T u_0>0$ can be replaced with the weaker 
condition $|\{u_0>0\}|>0$ without changing either \eqref{u>0}  or 
\eqref{KKMS}. Mueller \cite{Mueller1}  proved that the weaker condition 
$|\{u_0>0\}|>0$ suffices to imply \eqref{u>0}. Then, we first condition on 
$u(t_0)$ for a fixed value of $t_0>0$, and then use the Markov property and 
\eqref{u>0} to conclude that \eqref{KKMS} continues to hold when 
$|\{u_0>0\}|>0$.  } 
In this way, Condition \eqref{sigma(0)=0} assures us that
\begin{equation}\label{u>0}
	\P\left\{ u(t\,,x)>0\quad\text{for all $t>0$ and $x\in\T$}\right\}=1;
\end{equation}
see Mueller \cite{Mueller1}.

In a precursor to this paper (see \cite{KKMS}), together with S.-Y. Shiu
we proved the following.  If we additionally assume that
$\inf_{z\neq0}|\sigma(z)/z|>0$, then there exist non-random real numbers
$\lambda_1\ge\lambda_2>0$ such that 
\begin{equation}\label{KKMS}
        \P\left\{ \e^{-\lambda_1 t +o(t)} \le
        \inf_{x\in\T} u(t\,,x) \le \sup_{x\in\T}u(t\,,x)
        \le \e^{-\lambda_2 t +o(t)}
        \quad\text{as $t\to\infty$}\right\}=1.
\end{equation}
This proves that the solution to the stochastic heat equation \eqref{SHE}
dissipates precisely  exponentially as time increases. 
In this context, ``dissipation'' is another way to say that the solution tends to $0$ uniformly 
in the space variable $x$.

Let $U$ denote the solution to the non-random heat equation,
\[
	\partial_t U(t\,,x) = \partial^2_xU(t\,,x) \qquad
	\text{for $(t\,,x)\in(0\,,\infty)\times\T$,}
\]
subject to $U(0)=u_0$, where $u_0$ is the initial profile of the
SPDE \eqref{SHE}. 
We can write $U(t\,,x)=\int_{\T}p_t(x\,,y)u_0(y) \, \d y$ where $p_t(x\,,y)$ denotes the heat
kernel on $\T$; see \eqref{p} below for an expression for $p$. 
It is well known that, as $t\to\infty$, $p_t(x\,,y)\to 1/|\T|=1/2$
uniformly for $x$ and $y$ in $\T$.  Thus we find
$\lim_{t\to\infty}U(t\,,x) = \frac12\int_{\T}u_0(y)\,\d y>0$  uniformly for  $x\in\T$, 
in  contrast to \eqref{KKMS}. In other words, the dissipation result \eqref{KKMS} holds
in large part because the PDE \eqref{SHE} is forced randomly by the  white noise $\dot{W}$.

Here, we continue our analysis of \eqref{KKMS}, and introduce  methods
that yield the following almost-sure asymptotic result, which is the main result of this paper. 
The remainder of the paper is dedicated to the proof of this result. 
Before we state the result,
let us recall that the \emph{oscillation} of a function $f:\T\to\R$ is defined as
follows: For every relatively open set $X\subset\T$,
\[
	\text{\rm Osc}_{_X} (f) = \sup_{x\in X} f(x) - \inf_{y\in X} f(y)
	=\sup_{x,y\in X}|f(x)-f(y)|.
\]

\begin{theorem}\label{th:inf}
With probability one, $\text{\rm Osc}_{_\T}(\log u(t))\le(\log t)^{10+o(1)}$ as $t\to\infty$. 
\end{theorem}

\begin{remark}
We cannot reduce the exponent 10 in Theorem \ref{th:inf} with our present 
methods, and we do not know if this exponent is sharp.  
\end{remark}

Since $\inf_{x\in\T} \log u(t\,,x) + \log 2 \le
\log \|u(t)\|_{L^1(\T)} \le \sup_{x\in\T} \log u(t\,,x) + \log 2,$
it follows from Theorem \ref{th:inf} that, with probability one,
\[ 
	t^{-1} \log  \|u(t)\|_{L^\infty(\T)}  -t^{-1} \log \|u(t)\|_{L^1(\T)}\to 0  
	\quad\text{as $t\to\infty$.}
\]
The quantity $\|u(t)\|_{L^\infty(\T)} =\sup_{x\in\T}u(t\,,x)$ 
measures the size of the tallest peaks of $u(t)$, and
$\|u(t)\|_{L^1(\T)}=\int_\T u(t\,,x)\,\d x$ denotes the ``total mass'' at time $t$.
In this way we see that the tallest peaks and the total mass
almost surely have the same asymptotic behavior, to leading exponential order.
As it turns out, this implication can be effectively reversed:
We first prove  in \S\ref{sec:Lp} the following improvement of the preceding
display:
\begin{equation}\label{eq:Lpq}   
\log  \|u(t)\|_{L^\infty(\T)}  - \log \|u(t)\|_{L^1(\T)}
	=O (\log\log t) 
	\quad\text{as $t\to\infty$.}
\end{equation}
Then, in \S\ref{sec:proof_main} we 
appeal to a support  argument of Mueller \cite{Mueller1} in order to 
prove that \eqref{eq:Lpq} implies that
$\inf_{x\in \T} u(t\,, x)$ decays at the same exponential rate as $\|u(t)\|_{L^1(\T)}$.
The combination of these efforts establishes Theorem \ref{th:inf}.

Next, we make a few comments about the nature of 
the decay of the solution to \eqref{SHE}. 

Thanks to Theorem \ref{th:inf} and
our earlier collaboration with S.-Y. Shiu \cite{KKMS}
on the dissipation of parabolic SPDEs, we can see that
\begin{align*}
	-\infty  < \liminf_{t\to\infty} t^{-1} \inf_{x\in\T}\log  u(t\,,x)
		&=\liminf_{t\to\infty}  t^{-1}  \sup_{x\in\T}\log  u(t\,,x)<0,
		\quad\text{and}\\
	-\infty  < \limsup_{t\to\infty} t^{-1} \inf_{x\in\T}\log  u(t\,,x)
		&=\limsup_{t\to\infty} t^{-1} \sup_{x\in\T}\log  u(t\,,x) <0,
\end{align*}
provided additionally that $\inf_{z\neq0}|\sigma(z)/z|>0$.
It is natural to try and find conditions that ensure that the above $\liminf$s
and $\limsup$s are in fact bona fide limits. Such conditions would readily imply 
that there exists a number $\lambda>0$ such that 
\begin{equation}\label{lim}
	\adjustlimits\lim_{t\to \infty} \sup_{x\in\T}
	\left| t^{-1} \log u(t\,,x) + \lambda \right| = 0\qquad\text{almost surely}.
\end{equation}
That is, we would like to know when the entire solution dissipates at a precise
exponential rate. In the language of the literature on random media, \eqref{lim} says that
the solution to \eqref{SHE} has a \emph{uniform almost-sure Lyapunov exponent}
\cite{CM,CKM,CMS2002}. The earlier combined works of Carmona and Molchanov \cite{CM}
and Zeldovich, Molchanov, Ruzmaikin, and Sokolov \cite{ZMRS1985,ZMRS1988}
contain engaging discussions of the role of dissipation, and more generally
intermittency, in equations of random media from  mathematical and physical 
viewpoints, respectively.

We are able to use Theorem \ref{th:inf} in order
 to carry out this program in the special case of the 
\emph{parabolic Anderson model} only.

\begin{theorem}[A parabolic Anderson model]\label{th:PAM}
	Suppose in addition that there exists $Q>0$ such that $\sigma(z)=Qz$ for all $z\in\R$, and
	that $\inf_{x\in\T} u_0(x)>0$. 
	Then,  there exists a non-random real number
	$\lambda>0$ such that \eqref{lim} holds. Moreover, $\lambda$ does not depend on 
	the particular choice of $u_0$.
\end{theorem}

\begin{remark}\label{re:PAM}
	A day after the release of the first draft of this paper, Gu and Komorowski \cite[Proposition 4.1]{GK2}
	announced results that include the following striking formula:
	\[
		\lambda = \frac1\pi\left(\frac Q2\right)^6\e^{(\pi/Q)^2} \int_0^\infty
		\frac{\sinh(y)}{[\cosh(y/2)]^6}
		\sin\left( \frac{2\pi y}{Q^2}\right)\e^{-(y/Q)^2}\,\d y.
	\]
	[To apply Proposition 4.1 of \cite{GK2}, one sets their $\beta$ to our $Q$, and their $L$ to $2$.]
\end{remark}

There is a wide literature that implies the dissipation of the solution, particularly when:
\begin{compactenum}
	\item \eqref{SHE} is replaced by a similar SPDE on $\R_+\times\R$; 
	\item Almost sure convergence is replaced by another mode of convergence; and more significantly,
	\item Uniform convergence in \eqref{lim} is replaced with pointwise convergence.
\end{compactenum}
For example, Bertini and Giacomin \cite{BG} have studied \eqref{SHE}
on $\R_+\times\R$ in the special case that $\sigma(z)=Qz$ for all $z\in\R$
and $u_0 = \exp(B)$ where $B$ is a two-sided Brownian motion that is independent
of $\dot{W}$ [to ensure stationarity]. They proved that for every non random
$\varphi\in C^\infty_0(\R)$,
\[
	\frac1t\int_{-\infty}^\infty\varphi(y) \log u(t\,,y)\,\d y 
	\xrightarrow{L^2(\Omega)}
	-\frac{Q^4}{24}\int_{-\infty}^\infty\varphi(x)\,\d x\qquad
	\text{as $t\to\infty$}.
\]
Ideally, one would like to know that the above holds when $\varphi=\delta_x$ 
for an arbitrary $x\in\R$ and
with $L^2(\Omega)$-convergence replaced by almost sure convergence.
This would show that the [pointwise] 
\emph{almost sure Lyapunov exponent} of \eqref{SHE} is
$-Q^4/24$.
Amir, Corwin, and Quastel \cite{ACQ} considered the same SPDE as in \cite{BG},
but started at $u_0=\delta_0$, and proved among other interesting
things that indeed for every $x\in\R$ fixed,
\begin{equation}\label{ACQ}
	t^{-1} \log u(t\,,x) \xrightarrow\P -\frac{Q^4}{24}
	\quad\text{as $t\to\infty$.}
\end{equation}
The more recent work of Ghosal and Lin \cite{GL} implies that \eqref{ACQ} holds
for a wide class of initial data. Additional references can be found in a recent paper
by Gu and Komorowski \cite{GK}, where the asymptotics of the linear form of \eqref{SHE} is 
considered [that is, $\sigma(z)=Qz$ for all $z\in\R$],  
together with an associated central limit theorem,  
via a Feynman-Kac representation of a smoothed version of \eqref{SHE}
in spatial all spatial dimensions. The results of \cite{GK} in fact 
also imply the following, valid in the setting of the present paper: As $t\to\infty$,
\begin{equation}\label{GK}
	\frac{\log u(t\,,0) +\lambda t}{\sqrt t} \xrightarrow{\rm d\,}\eta,
\end{equation}
where $\lambda$ is the same constant that appeared in Theorem \ref{th:PAM} and Remark \ref{th:PAM},
and $\eta$ has a centered, non-degenerate normal distribution. When combined with Theorem
\ref{th:inf}, the result \eqref{GK} of Gu and Komorowski \cite{GK} immediately implies the
following surprising fact.
\begin{corollary}[A parabolic Anderson model]\label{th:PAM}
	Suppose in addition that there exists $Q>0$ such that $\sigma(z)=Qz$ for all $z\in\R$, and
	that $\inf_{x\in\T} u_0(x)>0$. Let $\eta$ and $\lambda$ be as in \eqref{GK}. Then, as $t\to\infty$,
	\[
		\frac{\log u(t) +\lambda t \mathbb{1}}{\sqrt t}
		\Rightarrow\eta\mathbb{1},
	\]
	where $\mathbb{1}(x)=1$ for all $x\in\T$, and ``$\Rightarrow$'' denotes 
	weak convergence in the space $C(\T)$.
\end{corollary}

There also are results with the desired almost sure convergence,
particularly when \eqref{SHE} is replaced with an SPDE on $\R_+\times\Z^d$, 
in which case $\partial^2_x$ is supplanted by the discrete Laplacian on $\Z^d$. Notably,
Carmona, Koralov, and Molchanov \cite{CKM}  and Cranston, Mountford, and Shiga 
\cite{CMS2002}  have proved independently and nearly at the same time
that $\Lambda(Q) = \lim_{t\to\infty} t^{-1}\log u(t\,,x)$
exists almost surely and is in $(-\infty\,,\infty)$ a.s. for every fixed $x\in\Z^d$.
Moreover, one can expect based on Ref.s \cite{CKM,CMS2002} that $\Lambda(Q)<0$
for all sufficiently large values of $Q$.\footnote{To be sure, 
	Carmona, Koralov, and Molchanov \cite{CKM}  and Cranston, Mountford, and Shiga 
	\cite{CMS2002} study the semi-discrete stochastic
	partial differential equation
	$\partial_t u = \kappa \Delta u+ u\dot{B}$ where $\kappa>0$, $\Delta$
	denotes the discrete Laplacian on $\Z^d$,
	and $\dot{B}$ is space-time white noise indexed by $\R_+\times\Z^d$. 
	Their results imply that $\Lambda(\kappa) =\lim_{t\to\infty}t^{-1}\log u(t\,,x)$ 
	exists almost surely,
	and satisfies $\Lambda(\kappa)<0$ for all sufficiently small $\kappa>0$.
	}

We pause to explore the sharpness of the linearity condition of Theorem \ref{th:PAM}.

\begin{OP}
	Let us first suppose that there exists $Q>0$ such that 
	$\sigma(z)=Qz$ for all $z\in\R$. In that case, we have seen that \eqref{lim} holds for some
	$\lambda=\lambda(Q)$, and one expects $\lambda(Q)$ to have nontrivial dependence 
	on $Q$.  For example, our recent work with S.-Y Shiu \cite{KKMS} proves this
	by showing that, if in addition $\inf_{z\neq0}|\sigma(z)/z|>0$, then
	$Q^{-4}\lambda(Q)$ is bounded  from above and below by positive constants,
	uniformly for all $Q\ge1$.
	Now suppose instead that there exist large numbers $Q_1\neq Q_2$
	and non-overlapping intervals
	$I_1,I_2,\ldots\subset\R$ with rapidly decreasing
	lengths such that
	$\sigma(z)=Q_1 z$ whenever $z\in I_{2n}$ for some $n\in\N$,
	and $\sigma(z)= Q_2 z$ when $z\in I_{2n+1}$ for some $n\in\N$.
	Then, Theorem \ref{th:PAM} intuitively suggests that, for a suitable choice of the intervals
	$I_1,I_2,\cdots$, we might expect $u(t)$  to decrease at an exponential rate 
	$\lambda(Q_1)$ for a while,  then switch to decaying at rate $\lambda(Q_2)$ for a while, then  back to 
	the decay rate $\lambda(Q_1)$, and so 
	on.  This heuristic argument implies that 
	the linearity condition of Theorem \ref{th:PAM} for $\sigma$
	is likely to be close to be  sharp. It should be possible to build on the quantitative assertions of this
	paper in order to make such counterexamples rigorous. However, that undertaking would 
	require a good deal more effort still. Because we are presently concerned with establishing
	positive results, we leave a rigorous construction as an open problem.
\end{OP}

Let us conclude the Introduction by
defining some notation that is used throughout.  
As is customary, we define $\lip(f)$ to be the optimal Lipschitz constant
of every real-valued function $f:I\to\R$, defined any subinterval $I$ of $\R$; that is,
\[
	\lip(f) = \sup_{\substack{x,y\in I\\x\neq y}}
	\frac{|f(x)-f(y)|}{|x-y|}.
\]
Thus, we can interpret the constant $\lip(\sigma)$ in \eqref{sigma(0)=0} as
the optimal such choice.

We occasionally let
\begin{equation}\label{IS}
	S_t = \sup_{x\in\mathbb{T}}u(t\,,x) = \|u(t)\|_{L^\infty(\T)}
	\qquad\text{for all $t\ge0$},
\end{equation}
in order to simplify the exposition.

Define $p_t(x\,,y)$ to be the heat kernel associated to the Laplace operator on $\T$
[tacitly endowed with periodic boundary conditions]. That is,
\begin{equation}\label{p}
	p_t(x\,,y) =p_t(x-y)= \frac{1}{\sqrt{4\pi t}}\sum_{n=-\infty}^\infty
	\exp\left( -\frac{(x-y+2n)^2}{4t} \right)\qquad\text{for all $t>0$ and $x,y\in\T$}.
\end{equation}
The heat kernel induces the transition semigroup $\{P_t\}_{t>0}$ of Brownian motion on $\T$,
defined via
\begin{equation}\label{P}
	(P_tu_0)(x) = \int_\T p_t(x\,,y) u_0(y)\,\d y
	\qquad\text{for $t>0$ and $x\in\T$},
\end{equation}
and all non-negative $u_0\in L^\infty(\T)$.  In this way,
we can give rigorous meaning to the stochastic  PDE \eqref{SHE}
using its mild formulation in the same manner as in  Walsh \cite{Walsh}. Namely,
\begin{equation}\label{mild}
	u(t\,,x) = (P_tu_0)(x) + \int_{(0,t)\times\T} p_{t-s}(x\,,y)
	\sigma(u(s\,,y))\, W(\d s\,\d y).
\end{equation}

For every real number $k\in[1\,,\infty)$, we always write
\[
	\|X\|_k = \{\E (|X|^k)\}^{1/k}
\]
for the $L^k(\Omega)$-norm of a random variable $X\in L^k(\Omega)$.

Throughout this paper, we follow the customary habit of writing $f \lesssim g$ when we mean that there 
exists a constant $C>0$ such that $f(x)\le Cg(x)$ for all $x$ in the stated range.  We write $f\gtrsim g$ 
iff $g\lesssim f$, and $f\asymp g$ is short-hand for the statement that both relations $f\lesssim g$
and $f\gtrsim g$ hold.

\section{Proof of Theorem \ref{th:PAM}}

Throughout this section,
we assume that the hypotheses of Theorem \ref{th:PAM} are met; that is,
\begin{equation}\label{s:PAM}
	u_0(x)=1
	\quad\text{and}\quad
	\sigma(z)=Qz\qquad\text{for all $x\in\T$ and $z\in\R$},
\end{equation}
where $Q>0$ is a fixed real number.
With this in mind,  we present the following which is the main result of this section.

\begin{theorem}\label{th:1}
	$\lim_{t\to\infty} t^{-1}\log \|u(t)\|_{L^\infty(\T)}$ exists and
	is in $(-\infty\,,0)$ almost surely.
\end{theorem}

It is possible to quickly present a conditional proof of Theorem \ref{th:PAM}, given that
we can establish Theorems \ref{th:inf} and \ref{th:1}. Therefore, let us dispense with the proof 
of Theorem \ref{th:PAM} first, conditionally on Theorems \ref{th:inf} and \ref{th:1}.
Then, we proceed to establish Theorem \ref{th:1}; that effort takes up the rest of this section. 
Theorem \ref{th:inf} will be proved subsequently.

\begin{proof}[Conditional proof of Theorem \ref{th:PAM}]
	If the initial data is $u_0\equiv 1$, then we combine Theorems
	\ref{th:inf} and \ref{th:1} in order to see that
	\[
		-\lambda = \lim_{t\to\infty} t^{-1} \sup_{x\in\T}\log u(t\,,x) 
		=\lim_{t\to\infty} t^{-1} \inf_{x\in\T}\log u(t\,,x), 
	\]
	almost surely, and $\lambda>0$ is non random. This yields \eqref{lim},
	provided that $u_0\equiv 1$.
	
	Next, suppose $u_0$ is a non-zero constant, say $\varkappa\neq1$. 
	Since $\varkappa^{-1}u$ solves \eqref{SHE} started identically from 1 [with $\sigma(z)=Qz$],
	the first portion of the proof implies 
	that $u$ satisfies \eqref{lim}, and the exponent $\lambda$ does not depend on $\varkappa$.

	Finally, let us suppose $u$ solves \eqref{SHE} with $\sigma(z)=Qz$ for all $z$ and
	$0<\inf_{\T} u_0\le \sup_{\T}u_0<\infty$.
	In accord with the comparison theorem for SPDEs
	(see for example Shiga \cite{Shiga}), 
	\[
		\underline{u} \le u \le \overline{u}\qquad\text{almost surely,}
	\]
	where $\underline{u}$
	and $\overline{u}$ respectively solve $\partial_t v = \partial^2_x v + Qv\dot{W}$ 
	with constant initial profiles
	$\underline{\varkappa}=\inf_\T u_0$ and 
	$\overline{\varkappa}=\sup_\T u_0$. Apply the preceding portions of the proof
	respectively to $\underline{u}$ and $\overline{u}$ in order
	to see that $u$ satisfies \eqref{lim},
	and the limiting exponent $\lambda$ does not depend on the initial data $u_0\in L^\infty(\T)$ 
	provided  that $\inf_\T u_0>0$. This completes the proof.
\end{proof}

The remainder of this section is devoted to proving Theorem \ref{th:1}.

In the case that \eqref{SHE} is replaced by the parabolic Anderson model on $(0\,,\infty)\times\Z^d$,
with $\partial^2_x$ replaced by the discrete Laplacian,
Carmona and Molchanov \cite{CM}  and Cranston, Mountford, and Shiga 
\cite{CMS2002} have shown that $\lim_{t\to\infty} t^{-1}\log u(t\,,x)$
exists a.s.\ for every $x\in\Z^d$. That is a quite similar result to the one announced in Theorem \ref{th:1},
but the results differ in two ways: First, Theorem \ref{th:1} is a statement about uniform convergence
and is not a pointwise assertion; and also significantly, unlike the previous
approaches of \cite{CM,CMS2002}, ours
cannot rely on a Feynman-Kac formulation of the solution since the Feynman-Kac expectation blows up
when the space variable is continuous. Instead, we use comparison arguments.
Still, as was done earlier in \cite{CM,CMS2002}, we prove the existence of a limit 
by appealing to Kingman's subadditive ergodic theorem for continuous-time processes;
see Kingman \cite[Theorem 4]{Kingman73}. 
Because the application of Kingman's theorem
in continuous time requires some care, as compared with the earlier discrete-time version of the
ergodic theorem of Kingman 
\cite{Kingman68}, we begin with a somewhat more general measure-theoretic discussion.

Let $(\Omega\,,\F,\P)$ be a complete probability space, and for every $t\ge0$ consider a mapping
$\vartheta_t:\Omega\to\Omega$ such that: 
\begin{compactenum}
	\item $\vartheta_0\omega=\omega$ for every $\omega\in\Omega$;
	\item $\vartheta_{t+s}=\vartheta_t\circ\vartheta_s$ for every $s,t\ge0$; and 
	\item Every $\vartheta_t$ preserves the measure $\P$;
		that is every $A\in\F$ has the same $\P$-measure as $\vartheta_t^{-1}A$
		for every $t\ge0$. 
\end{compactenum}
Then we say that $\vartheta=\{\vartheta_t\}_{t\ge0}$ is a \emph{measure-preserving semigroup}.

The following is  a continuous-time form of the Kingman subadditive ergodic theorem that has easy-to-check
conditions, but is otherwise a consequence of Kingman's original result
\cite[Theorem 4]{Kingman73} for continuous-time 
subadditive processes.\footnote{Kingman's theorem \cite[Theorem 4]{Kingman73} includes an additional technical condition
that, in the present context, translates to the assumption that $\inf_{t>0} t^{-1}\E X_t>-\infty$.
It is easy to see from Kingman's argument, however, that this condition is needed only in order
to prove that $\lim_{t\to\infty} t^{-1} X_t$ holds in $L^1(\Omega)$. We do not require it here,
as we are not interested in $L^1(\Omega)$-convergence.
}

\begin{proposition}\label{pr:Kingman}
	Let $\vartheta$ be as above and $X=\{X_t\}_{t\ge0}$ denote a real-valued stochastic process 
	that satisfies $X_{t+s}\le X_t+X_s\circ\vartheta_t$ a.s.\ for every $s,t\ge0$,
	as well as the following: There exist $c,a,b>0$ and $k>1$ such that
	\begin{equation}\label{eq:K12}
		\E\left( |X_t - X_s|^k\right)\le c |t-s|^{1+a}
		\quad\text{and}\quad
		\sup_{u\in[0,1]} \E\left( |X_u\circ \vartheta_r - X_u|^k\right) 
		\le c r^{1+b},
	\end{equation}
	uniformly for all $s,t,r\in[0\,,1]$. Then,
	$\lim_{t\to\infty} t^{-1}X_t = \inf_{n\in\N}\E( n^{-1}X_n\mid\mathcal{I})$ exists
	and is in $[-\infty\,,\infty)$ a.s.,
	where  $\mathcal{I}$ denotes the invariant $\sigma$-algebra
	$\{A\in\F:\, A=\vartheta_1^{-1}A\}$.
\end{proposition}

We include a proof for the sake of completeness.

\begin{proof}
	Let $x_{s,t} = X_{t-s}\circ \vartheta_s$ for all $0<s<t$, and observe that
	the two-parameter process $\{x_{s,t};\, 0<s<t<\infty\}$ is subadditive in the sense
	of Kingman \cite{Kingman73}. Clearly, \eqref{eq:K12} implies that
	\[
		\E\left( |X_t\circ\vartheta_s - X_h\circ\vartheta_s|^k\right)
		\le c |t-h|^{1+a}
		\quad\text{and}\quad
		\E\left( |X_u\circ\vartheta_s - X_u\circ\vartheta_r|^k\right) 
		\le c|s-r|^{1+b},
	\]
	uniformly for all $u,t,h,s,r\in[0\,,1]$.
	These bounds, and a suitable form of Kolmogorov's 
	continuity theorem \cite[Appendix C]{cbms}
	together ensure that the two-parameter
	process $\{x_{s,t};\, 0<s<t\}$ is continuous [up to a modification] and satisfies
	$\sup_{0<s<t<1}|x_{s,t}|\in L^1(\Omega)$.
	The proposition follows from
	Theorem 4 of Kingman \cite{Kingman73}.
\end{proof}

We now prepare to begin the proof of Theorem \ref{th:1}. Before we start, we 
need to deal with some
measure-theoretic issues first.

Recall that a two-parameter stochastic process $W=\{W(t\,,x)\}_{t\ge0,x\in\T}$ is a two-parameter
\emph{Brownian sheet} if $W$ is a centered Gaussian process and 
\[
	\Var\left(\int_{(0,t)\times\T}\phi\,\d W\right) = \int_0^t\d s\int_{\T}\d y\
	[\phi(s\,,y)]^2\qquad\text{for all $\phi\in L^2(\R_+\times\T)$ and $t>0$},
\]
where the integral $\int\phi\,\d W$ on the left-hand side denotes the Wiener integral that is associated
to the Brownian sheet $W$ (see Nualart \cite[Chapter 1]{Nualart}), and the integral on the 
right-hand side is Lebesgue's. It is well known that
$W$ has an almost surely continuous modification on $\R_+\times\T$. In this 
way, we may define the white noise $\dot{W}$ as the distributional space-time derivative of $W$;
that is, $\dot{W} = \partial_t\partial_xW$; see \v{C}entsov \cite{Centsov}.

Let $\Omega = C(\R_+\times\T)$ and endow $\Omega$ with its usual compact-open topology and 
associated Borel sigma-algebra $\F$. 
Let $\P$ denote the law of a two-parameter Brownian sheet.
Since the Brownian sheet is a.s.\ continuous, we may realize $\P$ as a probability measure on $\Omega$.
We may, and will, assume without loss of generality that
$\F$ is $\P$-complete. 

Let $W(t\,,x)(\omega) = \omega(t\,,x)$ for every $(t\,,x)\in\R_+\times\T$ and $\omega\in\Omega$;
that is $W(\omega)=\omega$ denotes the coordinate function on $C(\R_+\times\T)$.
Then, $W$ is a particular construction of a two-parameter Brownian sheet under the measure $\P$.
We may also introduce a measure-preserving semigroup $\vartheta=\{\vartheta_t\}_{t\ge0}$ 
on $(\Omega\,,\F,\P)$ as follows:
\begin{equation}\label{vartheta}
	(\vartheta_t\omega)(s\,,x) = \omega(t+s\,,x)-\omega(t\,, x)  \qquad\text{for all $s,t\ge0$,
	$x\in\T$, and $\omega\in\Omega$}.
\end{equation}
The following shows that the one-parameter stochastic process
$t\mapsto\log\|u(t)\|_{L^\infty(\T)}=\log S_t$ 
[see \eqref{IS}] satisfies the integrability property \eqref{eq:K12}
of our formulation of Kingman's subadditive ergodic theorem (Proposition \ref{pr:Kingman}).
Since Theorem \ref{th:1} can be proved
on any probability space, including our particular construction of $(\Omega\,,\F,\P)$,
Lemma \ref{lem:mod:S} reduces the proof of Theorem \ref{th:1} to the proof of
subadditivity property $X_{t+s}\le X_t+ X_s\circ\vartheta_t$, to which we return once we
verify the following.

\begin{lemma}\label{lem:mod:S}
	Suppose $u_0\equiv1$ and
	let  $X_t=\log S_t$ for every $t\ge0$. Then,
	for every $\alpha\in(0\,,1/8)$ and $k>1$ there exists a real number $c=c(k\,,\alpha)>0$
	such that
	\[
		\E\left( |X_t-X_s|^k\right) \le c |t-s|^{k\alpha}
		\quad\text{and}\quad
		\sup_{u\in[0,1]}\E\left(|X_u\circ\vartheta_r - X_u|^k \right)\le c 
		r^{k\alpha},
	\]
	uniformly for all $s,t,r\in(0\,,1)$.
\end{lemma}

\begin{remark}
	It is possible to refine the forthcoming argument to prove that the above in fact holds
	for every $\alpha\in(0\,,1/4)$; one uses H\"older's inequality in 
	\eqref{log-log} in place of the Cauchy-Schwarz inequality as is done here. We omit the 
	elementary details as we do not need the improvement.
\end{remark}

\begin{proof}
	Choose and fix a real number $k\ge2$. We start the proof by developing
	a preliminary estimate; see \eqref{:u-u:} below.
	
	Recall the following basic estimates from the literature
	\cite[Exercise 3.7, p.\ 323]{Walsh}:
	\begin{equation}\label{eq:mod-in-t}
		\|u(t\,,x)\|_k\lesssim1
		\quad\text{and}\quad
		\|u(t\,,x) - u(s\,,y)\|_k
		\lesssim\left( |s-t|^{1/4} + |x-y|^{1/2}\right),
	\end{equation}
	both valid uniformly for all $(s\,,y),(t\,,x)\in(0\,,1)\times\T$.
	In light of \eqref{s:PAM}, and according to Chapter 3 of Walsh \cite{Walsh}, 
	we may write the solution $u$ to \eqref{SHE}
	in the following mild form:
	\begin{equation}\label{mild:1}
		u(t\,,x) = 1 + Q\int_{(0,t)\times\T} p_{t-r}(y-x) u(r\,,y)\,W(\d r\,\d y),
	\end{equation}
	where the latter denotes the Walsh stochastic integral \cite[Chapter 3]{Walsh}. See also \eqref{mild}.
	Choose and fix some $s>0$.
	We can see immediately from the above, and from elementary facts about the Walsh stochastic integral,
	that the $\vartheta_s$-shift of $u$ has
	a mild representation.  To see this, let us first define
        \[
                v(t\,,x) = u(t\,,x)\circ\vartheta_s\qquad\text{for all $t\ge0$ and $x\in\T$}.
        \]
      	Then, $v$ solves 
	\[
		v(t\,,x) = 1 + Q\int_{(s,s+t)\times\T} p_{t+s-r}(y-x)
		v(r-s\,,y)\, W(\d r\,\d y)\qquad\text{for all $t\ge0$ and $x\in\T$}.
	\]
        Define $W_s(\d r\,\d y)=(W\circ\vartheta_s)(\d r\,\d y)$ in order to deduce
        from the preceding that
	\begin{equation}\label{v}
		v(t\,,x) = 1 + Q\int_{(0,t)\times\T} p_{t-r}(y-x)
		v(r\,,y)\, W_s(\d r\,\d y)\qquad\text{for all $t\ge0$ and $x\in\T$}.
	\end{equation}
	Because the Brownian sheet has stationary and independent increments in its first variable (say),
	it follows that $W_s$ is a Brownian sheet. 
	Thus, we see that $v$ is the solution to \eqref{SHE} but with $\dot{W}$ replaced
	by $\dot{W}_s$. Among other things, it follows that $v$  satisfies \eqref{eq:mod-in-t}.
	Therefore, whenever $0<s < t<1$ and $x\in\T$, we may write
	\[
		\left| u(t\,,x)\circ\vartheta_s - u(t\,,x) \right|
                = |v(t\,,x)-u(t\,,x)| 
                \le Q(q_1 + q_2 + q_3 + q_4 + q_5),
	\]
	where
	\begin{align*}
		q_1 &= \left| \int_{(0,s)\times\T} p_{t-r}(y-x)u(r\,,y)\,W(\d r\,\d y)\right|,\\
		q_2 &= \left| \int_{(t,s+t)\times\T} p_{t+s-r}(y-x)v(r-s\,,y)\,W(\d r\,\d y)\right|,\\
		q_3 &= \left| \int_{(s,t)\times\T}\left\{ p_{t+s-r}(y-x) - p_{t-r}(y-x) \right\}u(r\,,y)
			\, W(\d r\,\d y)\right|,\\
                q_4 &= \left| \int_{(s,t)\times\T}p_{t+s-r}(y-x)\left\{ u(r-s\,,y) -  u(r\,,y)\right\} W(\d r\,\d y)\right| \\
                q_5 &= \left| \int_{(s, t)\times\T}p_{t+s-r}(y-x)\left\{ v(r-s \,,y) -  u(r-s \,,y)\right\} W(\d r\,\d y)\right|.
	\end{align*}
	We estimate $q_1,\ldots,q_5$ next, and in this order. 
	
	A suitable application of the Burkholder-Davis-Gundy inequality
	\cite[Proposition 4.4]{cbms} yields
	\[
		\|q_1\|_k^2 \lesssim\int_0^s\d r\int_{\T}\d y\
		[ p_{t-r}(y-x)]^2 \|u(r\,,y)\|_k^2,
	\]
	where the implied constant depends neither on $s\in(0\,,t)$ nor on $(t\,,x)\in(0\,,1)\times\T.$ 
	Therefore, we may appeal to \eqref{eq:mod-in-t} in order to obtain
	\[
		\|q_1\|_k^2 \lesssim\int_0^s\d r\int_{\T}\d y\
		[ p_{t-r}(y-x)]^2
		=\int_0^s p_{2(t-r)}(0)\,\d r = \int_{t-s}^t p_{2r}(0)\,\d r,
	\]
	thanks to the semigroup property of the heat kernel. It is well known that
	$p_\tau(0)\lesssim \tau^{-1/2}$ uniformly for all $\tau\in(0\,,2]$;
	see for example Lemma B.1 of our earlier paper
	\cite[Appendix B]{KKMS}. This yields
	\[
		\|q_1\|_k^2 \lesssim  \int_{t-s}^t \frac{\d r}{\sqrt r}
		\propto \sqrt{t} - \sqrt{t-s}\lesssim\sqrt s,
	\]
	where the implied constants are independent of $0<s<t<1$ and $x\in\T$.
	Similarly,
	\[
		\|q_2\|_k^2 \lesssim\int_t^{s+t}\d r\int_{\T}\d y\
		[ p_{t+s-r}(y-x)]^2 \|v(r-s\, , y)\|_k^2  \lesssim \int_t^{s+t} p_{2(t+s-r)}(0)\,\d r = \int_0^s p_{2r}(0)\,\d r\lesssim\sqrt s,
	\]
	and another appeal to \eqref{eq:mod-in-t} yields
	\begin{align*}
		\|q_3\|_k^2 &\lesssim \int_s^t\d r\int_{\T}\d y
			\left\{ p_{t+s-r}(y-x) - p_{t-r}(y-x) \right\}^2\|u(r\,,y)\|_k^2\\
		&\lesssim \int_0^{t-s}\d r\int_{\T}\d y
			\left\{ p_{s+r}(y-x) - p_r(y-x) \right\}^2 \lesssim\sqrt s,
	\end{align*}
	all valid uniformly for $s\in(0\,,t)$ and $(t\,,x)\in(0\,,1)\times\T.$ 
	Another appeal to \eqref{eq:mod-in-t} yields the following bound:
	\begin{align*}
		\|q_4\|_k^2 &\lesssim \int_s^t\d r\int_{\T}\d y\ [p_{t+s-r}(y-x)]^2
			\| u(r-s\,,y) -  u(r\,,y)\|_k^2\\
		&\lesssim\sqrt s\int_s^t\d r\int_{\T}\d y\ [p_{t+s-r}(y-x)]^2 = 
			\sqrt s\int_0^{t-s}\d r\int_{\T}\d y\ [p_{r+s}(y-x)]^2\lesssim \sqrt s,
	\end{align*}
	also valid uniformly for all $s\in(0\,,t)$ and $(t\,,x)\in(0\,,1)\times\T$.
        
        Finally, we estimate $q_5$. 
        As before, a suitable application of 
        the Burkholder-Davis-Gundy inequality yields
	\[ 
		\|q_5\|_k^2 \lesssim \int_s^t \d r \int_{ \T} \d y\,  
		[p_{t+s-r}(y-x)]^2 \|v(r-s, y) - u(r-s, y)\|_k^2.
	\]
	Let $w(t\,,x)=v(t\,,x)-u(t\,,x)$ and combine the preceding observation
	to find that
	\begin{align*}
		\|w(t\,,x)\|_k^2 &\lesssim \sqrt{s} +
		\int_0^t\d s \int_{\T}\d y\ 
		[p_{t-s}(y-x)]^2\|w(s\,,y)\|_k^2,
	\end{align*}
	where the implied constant does not depend on $(t\,,x)\in(0\,,1)\times\T$.
        A Gronwall-type argument as in  Walsh \cite[Lemma 3.3]{Walsh}  now yields
        the bound,
        \[
		\|w(t\,, x)\|_k^2 \lesssim \sqrt{s},
        \]
        valid uniformly for all $t\in(0\,,1)$, $s\in(0\,,t)$ and $x\in\T$.
        Thus, we have 
	\begin{equation}\label{:u-u:}
		\left\| u(t\,,x)\circ\vartheta_s - u(t\,,x) \right\|_k \lesssim s^{1/4},
	\end{equation}
	valid uniformly for all $0<s<t<1$ and $x\in\T$. This is the preliminary estimate that was alluded to
	at the beginning of the proof. We can now establish Lemma \ref{lem:mod:S}.
	
	Since $|\log b-\log a|\le |b-a| (a^{-1}+b^{-1})$ for all $b,a>0$, 
	the Cauchy-Schwarz inequality plus the triangle inequality  together yield
	\begin{equation}\label{log-log}
		\|\log A - \log B\|_k \le
		\|B-A\|_{2k} \left( \|A^{-1}\|_{2k}+\| B^{-1}\|_{2k}\right),
	\end{equation}
	valid for all strictly positive random variables $A$ and $B$. According to the method of
	Mueller and Nualart \cite{MN}, 
	\begin{equation}\label{MN}
		c_k=\sup_{t\in(0,1)}\E\left( \inf_{x\in\T}|u(t\,,x)|^{-2k}\right)<\infty.
	\end{equation}
	Therefore, we may recall \eqref{IS} and let
	$X_t=\log S_t$ for every $t>0$ in order to deduce from
	the preceding remarks that
	\begin{equation} \label{eq:difference-X}
		\left\| X_t-X_s\right\|_k
		\le 2c_k^{1/(2k)}\left\| S_t-S_s \right\|_{2k},
        \end{equation}
	uniformly for all $0<s<t<1$. Since $k\ge 2$ can be as large
	as we would like, \eqref{eq:difference-X}, \eqref{eq:mod-in-t}, and Kolmogorov's continuity theorem
	together imply that for every fixed $\alpha\in(0\,,1/8)$,
        \begin{equation}\label{XX:1}
		\left\| X_t-X_s \right\|_k \lesssim |t-s|^\alpha
		\quad\text{uniformly for all $0<s<t<1$}.
	\end{equation} 
	Here we have used the fact that
	$|S_t-S_s| \le  \sup_{x\in\T}|u(t\,,x)-u(s\,,x)|.$
	The very same argument that led us to \eqref{XX:1} shows also that
	\[
		\|X_v-X_v\circ\vartheta_r\|_k \le 2c_k^{1/(2k)}
		\left\| S_v -  S_v\circ\vartheta_r \right\|_{2k}
		\qquad\text{for every $r,v\in[0\,,1]$}.
	\]
	Since the shift $\vartheta_r:\Omega\to\Omega$
	preserves the measure $\P$, \eqref{:u-u:}, Kolmogorov's continuity theorem, 
	and \eqref{MN} together imply that
	\begin{equation} \label{eq:diff-X-2}
		\sup_{v\in(0,1)}\left\| X_v-X_v\circ\vartheta_r \right\|_k \lesssim r^\alpha
		\qquad\text{for all $r\in[0\,,1]$}.
	\end{equation}
	Finally, \eqref{eq:diff-X-2} and \eqref{XX:1} together verify the 
assertions of Lemma \ref{lem:mod:S}.
\end{proof}

We are ready for the following.

\begin{proof}[Proof of Theorem \ref{th:1}]
	As has been mentioned already,
	we can write $u$ in mild form as in \eqref{mild:1}.
	Moreover, it is well known that the $C_{>0}(\T)$-valued stochastic
	process $\{u(t)\}_{t\ge0}$ is a strong Markov process; see
	Nualart and Pardoux \cite{NualartPardoux}. Here $C_{>0}(\T)$ 
        means that the process is continuous and takes values in the positive
        real numbers.  
	
	Now choose and fix an arbitrary number $s>0$. Then, elementary properties of the Walsh
	stochastic integral imply that for all $s,t\ge0$ and $x\in\T$,
	\begin{align*}
		u(t+s\,,x) &
			= (P_t u(s))(x) + Q\int_{(s,t+s)\times\mathbb{T}} p_{t+s-r}(x\,,y)u(r\,,y) \,W(\d r\,\d y)\\
		&=(P_t u(s))(x) + Q\int_{(0,t)\times\mathbb{T}} p_{t-r}(x\,,y)u(r+s\,,y) \,W_s(\d r\,\d y),
	\end{align*}
	almost surely,
	where $W_s(\d r\,\d y)=(W\circ\vartheta_s)(\d r\,\d y)$  defines a space-time white noise
	, and $\{P_t\}_{t\ge0}$ denotes the
	heat semigroup on $\mathbb{T}$. That is, $P_0f=f$ for all $f\in C(\T)$,
	and $(P_tf)(x) = \int_{\T}p_t(x\,,y)f(y)\,\d y$ for all $t>0$ and $x\in\T$.
	
	Recall the process $S=\{S_t\}_{t\ge0}$ from \eqref{IS}.
	The comparison theorem for SPDEs (see Shiga \cite{Shiga})
	and the independence of $W_s(\d r\,\d y)$ and $u(s)$ 
        together tell us that for all $s,t\ge0$ and $x\in\T$,
	\[
		u(t+s\,,x) \le u_s(t\,,x)\quad\text{and hence}\quad
		S_{t+s}\le\sup_{x\in\mathbb{T}}u_s(t\,,x)
		\qquad\text{a.s.},
	\]
	where $u_s$ solves 
	\[
		\partial_tu_s =  \partial_x^2 u_s + 
		Qu_s\dot{W}_s\quad\text{on $(0\,,\infty)\times\mathbb{T}$},
	\]
	subject to the initial profile $u_s(0)\equiv S_s$. Because the SPDE in question is linear,
	$v_s(t\,,x) = u_s(t\,,x)/S_s$ solves the SPDE,
	\[
		\partial_t v_s = \partial_x^2 v_s + 
		Qv_s\dot{W}_s\quad\text{on $(0\,,\infty)\times\mathbb{T}$},
	\]
	subject to $v_s(0)\equiv 1$. Because this SPDE has a unique solution,
	we compare the above with \eqref{v} (and recall the uniqueness of the SPDE
	that \eqref{v} describes) in order to deduce that
	$v_s=u\circ \vartheta_s$. Thus, it follows that
	\[
		S_{t+s} \le S_s \times (S_t\circ \vartheta_s)
		\quad\Rightarrow\quad
		\log S_{t+s} \le \log S_s + \log S_t\circ\vartheta_s\qquad\text{a.s.},
	\]
	This and Lemma \ref{lem:mod:S} together imply that $\{\log S_t\}_{t\ge0}$ satisfies
	the conditions of our formulation of Kingman's subadditive ergodic theorem
	(Proposition \ref{pr:Kingman}) and hence
	\[
		-\lambda=\lim_{t\to\infty} t^{-1} \log S_t
	\]
	exists a.s.\ and is measurable with respect
	to the invariant sigma-algebra of $\vartheta_1$. By the Kolmogorov 0-1 law,
	the latter sigma-algebra is trivial; therefore, $\lambda$ is non random. In principle,
	$\lambda$ could be any extended real number in $[-\infty\,,\infty)$. However, the theory of \cite{KKMS}
	implies, in the particular case that \eqref{s:PAM} holds, that
	\[
		-\infty<\liminf_{t\to\infty} t^{-1}\inf_{x\in\T}\log u(t\,,x)\le
		\limsup_{t\to\infty} t^{-1}\sup_{x\in\T}\log u(t\,,x) < 0\qquad\text{a.s.}
	\]
	This proves that $-\infty<\lambda<0$ and completes the proof.
\end{proof}

\section{An asymptotic interpolation theorem}\label{sec:Lp}
In this section, we return to the first part of the proof of Theorem \ref{th:inf}
and prove the following asymptotic  $L^p$-interpolation theorem.
This result will be used in the proof of Theorem \ref{th:inf} afterward. 
\begin{theorem}\label{th:S/M}
	Choose and fix two extended real numbers $1\le p,q\le\infty$. Then, a.s.,
	\[
		\log  \|u(t)\|_{L^p(\T)}  - \log \|u(t)\|_{L^q(\T)} =O(\log \log t)
		\quad\text{as $t\to\infty$}.
	\]
\end{theorem}
Jensen's inequality implies that Theorem \ref{th:S/M} is an equivalent
formulation of \eqref{eq:Lpq}.
Our proof of Theorem \ref{th:S/M} hinges on the following simple lemma.

\begin{lemma}\label{lem:WLOG}
        Theorem \ref{th:S/M} holds provided that, for every $\beta>3$,
	\begin{equation}\label{cond:WLOG}
		\sum_{n=1}^\infty
		\P\left\{ \sup_{0\le t\le n}\frac{\|u(t)\|_{L^\infty(\T)}}{\|u(t)\|_{L^1(\T)}}
		> (\log n)^\beta \right\} <\infty.
	\end{equation}
\end{lemma}
\begin{proof}
        According to \eqref{cond:WLOG} and the Borel-Cantelli lemma,
        with probability one,
	\begin{equation}  \label{eq:Borel-Cantelli}
		\sup_{0\le t\le n}\frac{\|u(t)\|_{L^\infty(\T)}}{\|u(t)\|_{L^1(\T)}}
		\le (\log n)^\beta \qquad\text{for all but a finite number of $n\in\N$}.
	\end{equation}
        We would like to replace $n$ in \eqref{eq:Borel-Cantelli} by a 
        continuous variable. Note that for $t\ge 1$ we have $ ([t]+1) \le 2t$
        where $[t]$ is the greatest integer in $t$.  So 
        \eqref{eq:Borel-Cantelli} implies that with probability 1,
	\[
		\frac{\|u(t)\|_{L^\infty(\T)}}{\|u(t)\|_{L^1(\T)}}
		\le \left( \log( [t]+1) \right)^\beta \leq  \left( \log t + \log 2 \right)^\beta 
		\qquad\text{for all sufficiently large $t>0$}.
	\]
	In other words, for every $\beta>3$
	\[
		\frac{ \log  \|u(t)\|_{L^\infty(\T)}  - \log \|u(t)\|_{L^1(\T)}}{\log\log t} \leq \beta+1 
		\quad\text{as $t\to\infty$ a.s.}
	\]
	Moreover, the inequality can be reversed since $\|u(t)\|_{L^1(\T)}\le 2\|u(t)\|_{L^\infty(\T)}$.
\end{proof}

The remainder of this section is devoted to the estimation of the probability
term in \eqref{cond:WLOG}. In order to do that, we must
overcome two challenges:
\begin{enumerate}
\item First, let us consider the non-random case $[\sigma\equiv0]$
	and take advantage of the following elementary property of the  heat semigroup
	$P=\{P_t\}_{t\ge0}$,
	defined earlier in \eqref{P}:
	\emph{$P$ tames very tall, thin peaks}. Thus, for example, 
	if $u_0\in C_+(\T)$ has a given area ---
	say $\|u_0\|_{L^1(\T)}=1$ ---
	and a much larger maximum ---
	say $\|u_0\|_{L^\infty(\T)}=N\gg1$ ---
	then $\|P_t u_0\|_{L^\infty(\T)} \ll N$ for relatively small values of $t$. Our first
	challenge is to show
	that the random heat operator $u\mapsto \partial_t u - \partial^2_x u - \sigma(u)\dot{W}$ 
	preserves essentially this taming property with high probability. 
	The details of this argument can be found in \S\ref{subsec:peaks} below.
\item Our second challenge is to prove that, with high probability,
	the total mass process $t\mapsto \|u(t)\|_{L^1(\T)}$
	does not get too big  or too small   in ``mesoscopic time,'' especially
	important when the initial data $u_0$ has very tall peaks in the sense of the previous paragraph. 
	This endeavor requires the simultaneous control of the total mass and the maximum of $u(t)$. 
	See \S\ref{subsec:total-mass} below for details.
\end{enumerate}
Once these challenges are met, we appeal to the strong Markov property 
of the infinite-dimensional process $\{u(t)\}_{t\ge0}$ in order to 
complete the proof of Theorem \ref{th:S/M}. This can be done relatively
effortlessly; see \S\ref{subsec:Pf:Th:S/M}.

\subsection{Control of tall peaks}\label{subsec:peaks}

We begin to work toward addressing our first challenge, mentioned in the preamble 
to this section.
Throughout, we denote the stochastic integral in \eqref{mild}
by $\mathcal{I}(t\,,x)$. That is,
\begin{equation}\label{I}
	\mathcal{I}(t\,,x)  = u(t\,,x) - (P_tu_0)(x) =\int_{(0,t)\times\T}
	p_{t-s}(x\,,y) \sigma(u(s\,,y))\, W(\d s\,\d y),
\end{equation}
for all $t>0$ and $x\in\T$. 

\begin{lemma}\label{lem:tp1}
	For every $T,\varpi>0$ there exists a real number $c = c(T,\varpi)>0$ such that
	\[
		\E\left( | \mathcal{I}(t\,,x) - \mathcal{I}(s\,,y) |^k\right) \le
		(ck)^{k/2}\exp( c k^3(s\vee t))\|u_0\|_{L^\infty(\T)}^k\left[ \sqrt{|t-s|} + |x-y| \right]^{k/2},
	\]
	uniformly for all Lipschitz-continuous functions $\sigma:\R\to\R$ that satisfy
	$\lip(\sigma)\le\varpi$,  $u_0\in L^\infty(\T)$, $x,y\in\T$, $0\le s,t\le T$, and $k\ge2$.
\end{lemma}

Lemma \ref{lem:tp1} is not a result about the solution to
\eqref{SHE} for a fixed diffusion coefficient $\sigma$; rather, it is a statement that
holds uniformly over all solutions to \eqref{SHE} for which the diffusion coefficient
satisfies $\lip(\sigma)\le\varpi$ and the initial profile is in $L^\infty(\T)$.

\begin{proof}
	It is well known that there exists a real number $A=A(\varpi)>1$ such that
	\begin{equation}\label{eq:mom:u}
		\sup_{x\in\T}\| u(t\,,x)\|_k  \le A \exp(Ak^2t)\|u_0\|_{L^\infty(\T)},
	\end{equation}
	uniformly for  all Lipschitz-continuous functions $\sigma:\R\to\R$ that satisfy
	$\lip(\sigma)\le\varpi$,  $u_0\in C_+(\T)$, $t>0$, and $k\ge2$;
	see \cite[Proposition 4.1]{KKMS} and its proof.
	
	Next we write, for all  $k\ge2$, $t,h>0$, and $x\in\T$,
	\[
		\| \mathcal{I}(t+h\,,x) - \mathcal{I}(t\,,x) \|_k \le J_1 + J_2
	\]
	where
	\begin{align*}
		J_1 &= \left\| \int_{(0,t)\times\T} 
			\left[ p_{t+h-s}(x\,,y) - p_{t-s}(x\,,y)\right] 
			\sigma(u(s\,,y))\,W(\d s\,\d y)\right\|_k,\\
		J_2 &= \left\| \int_{(t,t+h)\times\T}
			p_{t+h-s}(x\,,y)\sigma(u(s\,,y))\,W(\d s\,\d y)\right\|_k.
	\end{align*}
	We may estimate $J_1$ and $J_2$ in turn using the Burkholder-Davis-Gundy inequality as follows;
	see Khoshnevisan \cite[Proposition 4.4]{cbms} for the details of the application 
	of the Burkholder-Davis-Gundy inequality. For the same constant
	$A$ that appeared in \eqref{eq:mom:u},
	\begin{align*}
		J_1^2 &\le 4k|\lip(\sigma)|^2\int_0^t\d s\int_\T\d y
			\left[ p_{t+h-s}(x\,,y) - p_{t-s}(x\,,y)\right]^2\|u(s\,,y)\|_k^2\\
		&\le 4A^2\varpi^2 k\exp(2Ak^2t)\|u_0\|_{L^\infty(\T)}^2\int_0^t\d s\int_\T\d y
			\left[ p_{s+h}(x\,,y) - p_s(x\,,y)\right]^2.
	\end{align*}
	Therefore, we may apply Lemma B.6 of \cite{KKMS} in order to bound
	$\int_0^t\d s\int_\T\d y\ [ p_{s+h}(x\,,y) - p_s(x\,,y)]^2$ and deduce the following:
	\[
		J_1^2 \le 6A^2\varpi^2 k\exp(2Ak^2t) \|u_0\|_{L^\infty(\T)}^2\int_0^t
		\min\left( 1\,,\frac{h}{s}\right)\frac{\d s}{\sqrt{s}}.
	\]
	A direct evaluation of the integral [based on whether or not $s\le h$],
	and application of a square root, together yield
	\[
		J_1\le\sqrt{24k} A\varpi  \exp(Ak^2(t+h)) h^{1/4}\,\|u_0\|_{L^\infty(\T)}.
	\]
	Similarly, one obtains the following:
	\begin{align*}
		J_2^2 &\le 4k\varpi^2\int_t^{t+h}\d s\int_\T\d y \left[ p_{t+h-s}(x\,,y)\right]^2
			\|u(s\,,y)\|_k^2 \\
		&\le 4A^2\varpi^2k\exp(2Ak^2(t+h))\|u_0\|_{L^\infty(\T)}^2\int_0^h\d s\int_\T\d y
			\left[p_{s+h}(x\,,y)\right]^2 \\
		&=4A^2\varpi^2k\exp(2Ak^2(t+h))\|u_0\|_{L^\infty(\T)}^2\int_0^h
			p_{2(s+h)}(x\,,x)\,\d s.
	\end{align*}
	We now apply Lemma B.1 of our earlier paper with S.-Y. Shiu \cite{KKMS} in order to find that
	\begin{align*}
		J_2^2 &\le 8A^2\varpi^2k\exp(2Ak^2(t+h))\|u_0\|_{L^\infty(\T)}^2
			\int_0^h\max\left( \frac{1}{\sqrt{s+h}}\,,1\right)\d s\\
		&\le 8A^2\varpi^2k\exp(2Ak^2(t+h))\|u_0\|_{L^\infty(\T)}^2
			\int_0^h\frac{\d s}{\sqrt s} \\
		&= 16A^2\varpi^2k\exp(2Ak^2(t+h))\sqrt h\,\|u_0\|_{L^\infty(\T)}^2.
	\end{align*}
	Once again, the constant $A$ is the same that appeared in \eqref{eq:mom:u}.
	In this way, we find that
	\begin{equation}\label{II:1}
		\| \mathcal{I}(t+h\,,x) - \mathcal{I}(t\,,x) \|_k \le A\varpi
		\sqrt{80k}\,\exp(Ak^2(t+h)) h^{1/4}
		\|u_0\|_{L^\infty(\T)}.
	\end{equation}
	
	One proves, using similar arguments (see \cite[Lemma B.3]{KKMS})
	that for all $k\ge2$, $t\ge0$, and $x,z\in\T$,
	\begin{align*}
		\| \mathcal{I}(t\,,x) - \mathcal{I}(t\,,z) \|_k^2 &\le
			4A^2\varpi^2k\exp(2Ak^2t)\|u_0\|_{L^\infty(\T)}^2\int_0^t\d s\int_\T\d y
			\left[ p_s(x\,,y) - p_s(z\,,y) \right]^2\\
		&\le ck\exp(2Ak^2t)\|u_0\|_{L^\infty(\T)}^2|x-z|
			\int_0^t\frac{\d s}{s\wedge\sqrt s},
	\end{align*}
	for a real number $c>0$ that does not depend on $u_0$, $t>0$, $x,z\in\T$, nor $k\ge2$.
	Moreover, the constant $c$ depends on $\sigma$ only via $\varpi\ge\lip(\sigma)$.
	Thus, we find that
	\begin{equation}\label{II:2}
		\| \mathcal{I}(t\,,x) - \mathcal{I}(t\,,z) \|_k
		\lesssim \sqrt{k}\, \exp(Ak^2t)\|u_0\|_{L^\infty(\T)}
		\left( t^{1/4}\wedge \sqrt{\log_+ t}\right)|x-z|^{1/2},
	\end{equation}
	where $A$ is the same constant that appeared in \eqref{eq:mom:u},
	and the implied constant does not depend on $t>0$, $x,z\in\T$,  $k\ge2$,
	or $u$, except that $u_0\in L^\infty(\T)$ and $\lip(\sigma)\le\varpi$.
	Lemma \ref{lem:tp1} follows from \eqref{II:1} and \eqref{II:2}, and 
the triangle inequality.
\end{proof}

Lemma \ref{lem:tp1} has the following consequence for the stochastic integral
process $\mathcal{I}$ from \eqref{I}. Recall that $\varpi$ serves as
a proxy for an upper bound for the Lipschitz 
constant for $\sigma$.  

\begin{lemma}\label{lem:tp2}
	For every $\varpi>0$ and $0<\theta<\frac14$
	there exists a number $c=c(\theta\,,\varpi)>0$ such that 

	\begin{align*}
		\E\left(  \| \mathcal{I}(t) \|_{L^\infty(\T)}^k\right) 
			&\le (ck)^{k/2}\exp(ck^3t)
			\|u_0\|_{L^\infty(\T)}^k t^{k/4},\\
	\E\left(  \sup_{s\in[0,t]}\| \mathcal{I}(s) \|_{L^\infty(\T)}^k\right)
			&\le (ck)^{k/2}\exp(ck^3t)
			\|u_0\|_{L^\infty(\T)}^k t^{k\theta}, 
	\end{align*}
	uniformly for $t\in(0,1]$ and for
	all Lipschitz-continuous functions $\sigma:\R\to\R$ that satisfy
	$\lip(\sigma)\le\varpi$,   $u_0\in L^\infty(\T)$, and $k\ge2$. 
\end{lemma}

\begin{proof}
	The first inequality comes from Lemma \ref{lem:tp1} and a quantitative 
	form of Kolmogorov continuity theorem (see Khoshnevisan \cite[Appendix C]{cbms}). 
	The second inequality also comes from Lemma \ref{lem:tp1} and the Kolmogorov continuity theorem.
	We will mention how, since this sort of argument can arise multiple times:
	\begin{align*}
		| \mathcal{I}(s\,, x)|   &\leq t^\theta \left[ \frac{|\mathcal{I}(s\,, x) - 
			\mathcal{I}(0\,, x)|}{s^\theta} \right]  + |\mathcal{I}(0\,,x)|\\
		&\leq  t^\theta 
			\left[ \sup_{\substack{ (s_1, x_1) \neq (s_2, x_2)  \\ s_1, s_2\in [0, t], \, x_1, x_2\in\T}} 
			\frac{|\mathcal{I}(s_1, x_1)-\mathcal{I}(s_2, x_2)|}{|s_1-s_2|^\theta + |x_1-x_2|^{2\theta} }  
			\right]  + |\mathcal{I}(0\,,x)|,
	\end{align*}
	uniformly  for all $s\in [0\,, t]$ and $x\in \T$. This completes the proof.
\end{proof}

\begin{lemma}\label{lem:pr2}
	For every $\varpi>0$ and $0<\theta<\frac14$ 
there exists
	$c=c(\theta\,,\varpi)>0$ such that
	\begin{align*}
		\left\| \|u(t)\|_{L^\infty(\T)} \right\|_k
			&\lesssim \sqrt{k}\,\exp( c k^2t) \|u_0\|_{L^\infty(\T)} t^{1/4}
			+   t^{-1/2},\\
		\left\| \sup_{s\in(0,t)}\|u(s)\|_{L^\infty(\T)} \right\|_k
			&\le c\sqrt{k}\,\exp( c k^2t) \|u_0\|_{L^\infty(\T)} t^\theta
			+ \|u_0\|_{L^\infty(\T)},
	\end{align*}
	uniformly for all $k\ge2$, $0<t\le1$, 
	all Lipschitz-continuous functions $\sigma:\R\to\R$ that satisfy
	$\lip(\sigma)\le\varpi$,  and all $u_0\in C_+(\T)$ that satisfy
	$\|u_0\|_{L^1(\T)}=1$.
\end{lemma}

\begin{proof}
	Because $\sup_{x\in\T}p_t(x)\le 2(1\vee t^{-1/2})$ for all $t>0$
	(see for example \cite[Lemma B.1]{KKMS}), and since $\int_\T u_0(x)\,\d x=1$,
	it follows that $\|P_tu_0\|_{L^\infty(\T)}\le 2(1\vee t^{-1/2})$.
	The first portion of the lemma follows from this, Lemma \ref{lem:tp2}, and 
	\eqref{I}. The second portion follows similarly.
\end{proof}

We are ready to present and prove the main result of this subsection.
The following device controls the tall peaks of the solution,
and addresses the first of the two challenges that were mentioned
earlier on in the section.

\begin{proposition}\label{pr:KPZ1}
	For every $\varpi>0$ and $\frac43<\gamma<2$  there exist $K=K(\gamma\,,\varpi)>0$ 
	such that
	\begin{align*}
		\P\left\{ \left\| u\left( N^{-\gamma} \right)
		\right\|_{L^\infty(\T)}  \ge K N^{\gamma/2} \right\}
		&\le K\exp\left( - N^{(3\gamma-4)/2} \right),\\
		\P\left\{  \sup_{0\le s\le N^{-\gamma} }\|u(s)\|_{L^\infty(\T)} \ge 
		2N\right\} & \le K\exp\left( - \tfrac12 N^{(3\gamma-4)/2} \right),
	\end{align*}
	uniformly for all real numbers $N\ge1$, 
	all Lipschitz-continuous functions $\sigma:\R\to\R$ that satisfy
	$\lip(\sigma)\le\varpi$,  and all $u_0\in C_+(\T)$
	that satisfy $\|u_0\|_{L^1(\T)}=1$ and $\|u_0\|_{L^\infty(\T)}\le N$.
\end{proposition}

The proof of Proposition \ref{pr:KPZ1} rests solely on Lemma \ref{lem:pr2} and 
Chebyshev's inequality. Still, there are a number of parameters that needs to be
controlled and the ensuing ``numerology'' is slightly messy. Therefore, we include
some of the requisite details in order to help with the perusal of the argument.

\begin{proof}
	Since $\gamma <2$, we can  choose and fix $0<\theta<\frac14$ that satisfies $\gamma(3-4\theta) < 4$.
	We apply Lemma \ref{lem:pr2} with $t=N^{-\gamma}$ 
	and $k= N^{(3\gamma-4)/2}$ in order to see that there exists a
	real number $K_1=K_1(\gamma\,,\theta\,,\varpi)>0$ such that
	\begin{equation}\label{Joob}\begin{split}
		\left\| \|u(N^{-\gamma})\|_{L^\infty(\T)} \right\|_{N^{(3\gamma-4)/2}}
			&\le \e^{-1} K_1N^{\gamma/2},\\
		\left\| \sup_{0\le s\le N^{-\gamma}}\|u(s)\|_{L^\infty(\T)} \right\|_{N^{(3\gamma-4)/2}}
			&\le N + \tfrac12 K_1N^{\gamma(3-4\theta)/4},
	\end{split}\end{equation}
	uniformly for all $N\ge 2^{2/(3\gamma-4)}$,
	all Lipschitz-continuous functions $\sigma:\R\to\R$ that satisfy
	$\lip(\sigma)\le\varpi$,  and all $u_0\in C_+(\T)$ that satisfy
	$\|u_0\|_{L^1(\T)}=1$. This is because $\exp(ck^2/N^\gamma)$ is bounded uniformly
	in $(k\,,N)$ for the present choices of $(k\,,N)$. 
	The first bound in \eqref{Joob} and Chebyshev's inequality together imply that
	\begin{align*}
		\P\left\{ \left\| u\left( N^{-\gamma} \right)
			\right\|_{L^\infty(\T)}  \ge K_1 N^{\gamma/2} \right\}
			&\le  \E\left( \left| \frac{\left\| u\left( N^{-\gamma} \right)
			\right\|_{L^\infty(\T)}}{K_1N^{\gamma/2}}\right|^k\right) \\
		&\le \exp\left( - N^{(3\gamma-4)/2} \right)\qquad\text{for
			all $N\ge 2^{(3\gamma-2)/2}$}.
	\end{align*}	
	Thus, we find that there exists $K_2=K_2(\gamma\,,\theta\,,\varpi)>0$ such that
	\[
		\P\left\{ \left\| u\left( N^{-\gamma} \right)
		\right\|_{L^\infty(\T)}  \ge K_1 N^{\gamma/2} \right\}
		\le K_2\exp\left( - N^{(3\gamma-4)/2} \right)\qquad\text{for
		all $N\ge 1$},
	\]
	which is another way to state the first assertion of the proposition.
	
	In order to deduce the second portion of the proposition, note that the constant
	$\gamma(3-4\theta)/4$ [that appears in the exponent of $N$ in the second part of \eqref{Joob}]
	lies strictly between $0$ and $1$.
	Therefore, it follows from \eqref{Joob} that for every $q>1$
	there exists $N_0=N_0(q\,,\gamma\,,\theta\,,\varpi)>0$
	such that
	\[
		\left\| \sup_{0\le s\le N^{-\gamma}}\|u(s)\|_{L^\infty(\T)} \right\|_{N^{(3\gamma-4)/2}}
		\le qN\qquad\text{for all $N\ge N_0$}.
	\]
	Apply this with $q=2\exp(-1/2)$, and then use Chebyshev's inequality to deduce that
	\[
		\P\left\{ \sup_{0\le s\le N^{-\gamma}}\|u(s)\|_{L^\infty(\T)} 
		\ge 2N \right\}\le \exp\left( -\tfrac12 N^{(3\gamma-4)/2}\right)\qquad\text{for all $N\ge N_0$}.
	\]
	Consequently, there exists $K_3=K_3(\gamma\,,\theta\,,\varpi)>0$ such that

	\[
		\P\left\{ \sup_{0\le s\le N^{-\gamma}}\|u(s)\|_{L^\infty(\T)} 
		\ge 2N \right\}\le 
		K_3\exp\left( -\tfrac12 N^{(3\gamma-4)/2}\right)\qquad\text{for all $N\ge1$}.
	\]
	The proposition follows with $K=\max(K_1\,,K_2\,,K_3)$.
\end{proof}

\subsection{Control of total mass}\label{subsec:total-mass}

We now turn to the second-mentioned challenge of the section. The solution to that challenge
lies in the next proposition. Specifically, the following asserts that,
on a mesoscopic time scale, there is a high probability that the 
total mass of the solution is
not unduly small, even if the solution starts out with a rather tall peak at time
zero (and is everywhere else small at that time), and hence the spatial maximum of 
the noise coefficient may be large to begin with.

\begin{proposition}\label{pr:valleys}
	For every $\varpi>0$ and $\frac43<\gamma<2$, there exists $L=L(\gamma\,,\varpi)>1$ such that
	\[
		\P\left\{ \inf_{0\le t\le N^{-\gamma}} \|u(t)\|_{L^1(\T)} 
		\le  \frac{1}{2} \quad \text{or}\   \sup_{0\le t\le N^{-\gamma}}
		\|u(t)\|_{L^1(\T)} \ge  2 \right\}
		\le L\exp\left( -\frac{N^{(3\gamma-4)/3}}{L} \right),
	\]
	uniformly for all $N\ge2$,
	all Lipschitz-continuous functions $\sigma:\R\to\R$ that satisfy
	$\lip(\sigma)\le\varpi$,
	and all $u_0\in C_+(\T)$ that satisfy $\|u_0\|_{L^1(\T)}=1$ and $\|u_0\|_{L^\infty(\T)}\le N$.
\end{proposition}

We will prove Proposition \ref{pr:valleys} after we comment on a subtle feature of that proposition.
Define
\[
	\mathfrak{M}_t = \|u(t)\|_{L^1(\T)} = \int_\T u(t\,,x)\,\d x\qquad\text{for all $t\ge0$}.
\]
This is the total mass process of the solution to \eqref{SHE}, and is well known to be
a nice martingale. In fact, we can integrate both sides of \eqref{I} and appeal to a stochastic Fubini theorem
in order to conclude the well-known fact that
\[
	\mathfrak{M}_t = \mathfrak{M}_0 + 
	\int_{(0,t)\times\T} \sigma(u)\,\d W\qquad[t>0].
\]
Thus, we see that $\{\mathfrak{M}_t\}_{t\ge0}$ is a continuous $L^2(\Omega)$-martingale
with quadratic variation given by
\[
	\< \mathfrak{M}\>_t = \int_0^t\|\sigma(u(s))\|_{L^2(\T)}^2\,\d s
	\qquad[t\ge0].
\]
An appeal to a suitable form of the Burkholder-Davis-Gundy inequality and \eqref{sigma(0)=0} yields
the following: For all $t\ge0$ and $k\ge2$,
\[
	\| \mathfrak{M}_t-\mathfrak{M}_0\|_k^2 \le 4k\int_0^t\d s\int_\T\d y\ \|\sigma(u(s\,,y))\|_{2k}^2
	\le 4k|\lip(\sigma)|^2\int_0^t\d s\int_\T\d y\ \|u(s\,,y)\|_{2k}^2;
\]
see Khoshnevisan
\cite[Proposition 4.4]{cbms}. Therefore, \eqref{eq:mom:u} yields the following: For every $k\ge2$,
\begin{equation}\label{M:conc}
	\| \mathfrak{M}_t-1\|_k \lesssim  N\sqrt t,
\end{equation}
uniformly for all $t\ge0$, $N\ge1$, and $u_0\in L^\infty(\T)$ such that $\|u_0\|_{L^1(\T)}=1$ and
$\|u_0\|_{L^\infty(\T)}\le N$.
This bound turns out to be essentially unimprovable. 

Instead of establishing the above assertions in great detail, let us simply apply
them in order to be able to observe that $\P\{\mathfrak{M}_t\approx1\}\approx1$ provided that the
time variable $t$ is measured on a microscopic scale: $t\ll N^{-2}$.
Among other things, this shows that a well-known concentration estimate such as \eqref{M:conc}
yields $\P\{\mathfrak{M}_t\approx0\}\approx0$ when $t\ll N^{-2}$.
Proposition \ref{pr:valleys} says that we still have $\P\{\mathfrak{M}_t\approx0\}\approx0$ even
when $t$ is in the mesoscopic scale, $N^{-2} \ll t \ll 1$. But,
as we shall see, this is true for more subtle reasons than a mere concentration 
fact such as \eqref{M:conc}.

\begin{proof}[Proof of Proposition \ref{pr:valleys}]
	Concentration estimates such as \eqref{M:conc} fail to prove Proposition \ref{pr:valleys}
	because there is a very tall peak at time 0. That is, when $\|u_0\|_{L^1(\T)}=1$
	yet $\|u_0\|_{L^\infty(\T)}=N\gg1$ one is faced with an ``intermittency effect.'' That
	renders a moment bound such as \eqref{M:conc} useless. Therefore, in order to subdue
	the ``intermittency effect'',  we  control the quadratic variation of $\mathfrak{M}_t$ by 
	using $\|u(t)\|_{L^\infty(\T)}$ and $\|u(t)\|_{L^1(\T)}$ as follows. Define 
	\[ 
		\mathfrak{N}_t = \mathfrak{M}_t-1 
		  = \int_{(0,t)\times \T} \sigma(u(s\,, y)) \, W(\d s\, \d y)
		  \hskip1in[t\ge0].
	\]
	
	We may write 
	\begin{align*}
		\P\left\{  \inf_{0\le t\le N^{-\gamma}} \|u(t)\|_{L^1(\T)} \le  
			\frac{1}{2} \quad \text{or} \   \sup_{0\le t\le N^{-\gamma}} \|u(t)\|_{L^1(\T)} \ge  2\right\}
			& \leq \P \left\{  \sup_{0\le t\le N^{-\gamma}} 
			| \mathfrak{N}_t| \geq \frac{1}{2} \right\} \\
		&\leq \P(A_1) + \P(A_2) + \P(A_3), 
	\end{align*}
	where $A_1$, $A_2$, and $A_3$ are events that are defined in reverse order as follows:
	\begin{align*}
		A_3 &=   \left\{\omega\in\Omega:\,
			\sup_{0\le t\le N^{-\gamma}} \|u(t)\|_{L^1(\T)} (\omega)
			\geq N^{1/3} \right\};\\
		A_2& =   \left\{  \omega\in\Omega:\,
			\sup_{0\le t\le N^{-\gamma}}  \|u(t)\|_{L^\infty(\T)}
			(\omega) \geq  2N \right\};\\
		A_1& =   \left\{  \omega\in\Omega:\,
			\sup_{0\le t\le N^{-\gamma}} | \mathfrak{N}_t(\omega)| \geq 
			\frac12  \right\}\cap A_2^c\cap A_3^c. 
	\end{align*}
	Proposition \ref{pr:KPZ1} tells us  that there exists a number 
	$K=K(\gamma\,,\varpi)>0$ such that 
	\[ 
		\P(A_2) \leq K\exp\left( - \tfrac12 N^{(3\gamma-4)/2} \right),
	\]
	uniformly for all $N\ge2$.
	
	Next, we consider the event $A_1$. Since  $\mathfrak{N}$ is a continuous 
	$L^2(\Omega)$-martingale with mean zero and quadratic variation 
	\[
		\langle \mathfrak{N}\rangle_t = \int_{(0,t)\times\T} 
		[\sigma(u(s\,, y))]^2 \, \d y\,  \d s
		\qquad[t\ge0].
	\]
	Therefore, almost surely on the event $A_2^c\cap A_3^c$,
	\begin{equation} \label{eq:quad-bd}
		\sup_{0\le t\le N^{-\gamma}}  \langle \mathfrak{N}\rangle_t \leq [\lip(\sigma)]^2  
		\sup_{0\le t\le N^{-\gamma}}\int_{(0,t)\times\T} [u(s\,, y)]^2
		\, \d y\,  \d s \leq  2 \varpi^2 \, N^{-(3\gamma-4)/3},
	\end{equation}
	uniformly for all solutions of the SPDE \eqref{SHE} as long as $\lip(\sigma)\le\varpi$.
	Thanks to the martingale representation theorem, there exists a Brownian motion 
	$B=\{B(t)\}_{t\ge0}$ such that
	$\mathfrak{N}_t = B(\langle \mathfrak{N}\rangle_t)$ for all $t\ge0$.
	Therefore, the reflection principle and \eqref{eq:quad-bd} together
	imply that there exists a real number $c_1=c_1(\varpi\,, \gamma)>0$ such that 
	\begin{align*} 
		\P(A_1)& \leq  \P \left\{ \sup_{0\le t\le N^{-\gamma}} 
			\left| B(\langle \mathfrak{N}\rangle_t) \right| 
			\geq \frac12\ ,  \sup_{0\le t\le N^{-\gamma}} \langle 
			\mathfrak{N}\rangle_t \leq 2 \varpi^2 \, N^{-(3\gamma-4)/3}   \right\} \\ 
		&\leq \P\left\{  \sup_{0\le t\le 2 \varpi^2 \, N^{-(3\gamma-4)/3} }  |B(t)| \geq \frac12   \right\}
			\leq 2 \exp\left(  - c_1 N^{(3\gamma-4)/3} \right),
	\end{align*}
	uniformly for all $N\ge2$. We pause to mention that
	the above basically reproduces Freedman's martingale inequality
	\cite{Freedman}.
	
	Finally,  we bound $\P(A_3)$.  By a suitable application of the Burkholder-Davis-Gundy inequality
	\cite[Proposition 4.4]{cbms}, for all $k\geq 2$ and $t\ge0$,
	\begin{align*}
		\E \left(\left|  \mathfrak{N}_t \right|^k \right) 
			&\leq (4k)^{k/2}\E\left( \<\mathfrak{N}_t\>^{k/2}\right)
			= (4k)^{k/2} \E  \left( \left| \int_{(0,t)\times\T} [\sigma(u(s\,, y))^2
			\, \d s\, \d y \right|^{k/2}\right)\\
		&\leq (4k)^{k/2}  \varpi^k  \left( \int_{(0,t)\times\T}
			\|u(s\,, y)\|_k^{2} \,\, \d s\, \d y\right)^{k/2}
			\leq (3A\varpi N)^{k}\, (tk)^{k/2}\,  \exp\left( Ak^3 t\right).
	\end{align*}
	In the last inequality, we have used the fact that $\|u_0\|_{\infty}\leq N$ together with
	\eqref{eq:mom:u}.
	And the number $A$ is the constant that appeared in \eqref{eq:mom:u} and is, in particular,
	independent of $N\ge2$ and $t\ge0$. Whenever $N\ge2$,
	\[
		N^{1/3} - 1 \ge 
		N^{1/3} - (N/2)^{1/3} = C^{-1} N^{1/3},
	\]
	with $C = (1 - 2^{-1/3})^{-1}$.
	Therefore, Doob's inequality ensures that
	\begin{align*}
		\P(A_3) &= \P\left\{ \sup_{0\le t\le N^{-\gamma}} |\mathfrak{N}_t + 1|  \ge N^{1/3}\right\}
			\le \P\left\{ \sup_{0\le t\le N^{-\gamma}}  |\mathfrak{N}_t| 
			\geq C^{-1} N^{1/3} \right\}\\
		&\leq C^k N^{-k/3}\,  \E  \left( |\mathfrak{N}_{N^{-\gamma}}|^k \right)
			\leq (3CA\varpi)^{k} k^{k/2}  N^{-k(3\gamma-4)/6}\, \exp(Ak^3 N^{-\gamma})\\
		&= \left( \frac{(3CA\varpi)^2 k}{N^{(3\gamma-4)/3}}\right)^{k/2} \exp(Ak^3 N^{-\gamma}),
	\end{align*}	
	for all $k,N\ge2$. We can choose 
	\[
		k = \frac{N^{(3\gamma-4)/3}}{(3CA\varpi\e)^2}
	\]
	in order to see that
	\[
		\P(A_3) \le \exp\left( -k + Ak^3 N^{\gamma}\right) 
		= \exp\left(  -\frac{N^{(3\gamma-4)/3}}{(3CA\varpi\e)^2}  
		+ \frac{AN^{2\gamma-4}}{(3CA\varpi\e)^6}\right).
	\]
	Because $\gamma<2$, it follows that $2\gamma-4<(3\gamma-4)/3$, and this means 
	that the negative part of the exponent dominates. Therefore,
	there exists a number $C_0=C_0(\varpi\,,\gamma)>0$ such that
	\[
		\P(A_3) \le C_0^{-1}\exp\left( -C_0N^{(3\gamma-4)/3}\right).
	\]
	We may combine our estimates for the respective probabilities of
	$A_1,A_2,A_3$ in order to find that, uniformly for all $N\ge2$,
	\begin{align*}
		&\P\left\{  \inf_{0\le t\le N^{-\gamma}} \|u(t)\|_{L^1(\T)} \le  
			\frac{1}{2} \quad \text{or} \   \sup_{0\le t\le N^{-\gamma}} \|u(t)\|_{L^1(\T)} \ge  2\right\}
			\leq \P(A_1)+\P(A_2)+\P(A_3) \\
		&\leq 2 \exp\left(  - c_1 N^{(3\gamma-4)/3} \right) +
			K\exp\left( - \tfrac12 N^{(3\gamma-4)/2} \right)
			 + C_0^{-1}\exp\left( -C_0N^{(3\gamma-4)/3}\right),
	\end{align*}
	to conclude the proof from the elementary fact that
	$(3\gamma-4)/3 < (3\gamma-4)/2$.
\end{proof}

\subsection{Completion of the proof of Theorem \ref{th:S/M}}\label{subsec:Pf:Th:S/M}

The proof relies on a few applications of the strong Markov property. With the latter in mind,
let $\mathcal{F}(W)=\{\mathcal{F}_t(W)\}_{t\ge0}$ denote the Gaussian filtration that is
generated by the white noise $\dot{W}$. 
It might help to recall that one constructs these
sigma-algebras as follows: First, for every $t>0$ we 
let $\mathcal{F}_t^0$ denote the sigma-algebra that is generated
by all Wiener integrals of the form $\int_{(0,t)\times\T}\phi\,\d W$
as $\phi$ ranges over $L^2(\T)$.
Then, we $\P$-complete every $\mathcal{F}_t^0$, call the completion $\mathcal{F}_t^1$, and finally
we make these right-continuous and call the resulting sigma-algebra $\mathcal{F}_t$; that is,
$\mathcal{F}_t(W) = \cap_{s>t}\mathcal{F}_s^1$ for all $t\ge0$.

Fix $\beta>3$. For every $n\ge1$ define
\[
	\tau(n) = \inf\left\{ t>0:\, \| u(t)\|_{L^\infty(\T)} \ge  
	(\log n)^{\beta/2} \| u(t)\|_{L^1(\T)}\right\}
	\qquad[\inf\varnothing=\infty].
\] 
The general theory of processes ensures that every $\tau(n)$ is a stopping time with respect to
the Gaussian filtration $\mathcal{F}(W)$; see Bass \cite{Bass} for
a modern account. With Lemma \ref{lem:WLOG} in mind, we then write
for every $n\ge1$,
\begin{align*}
	\P\left\{ \sup_{0\le t\le n}\frac{\|u(t)\|_{L^\infty(\T)}}{
		\|u(t)\|_{L^1(\T)}} \ge (\log n)^{\beta}\right\} &=
		\P\left\{ \tau(n)<n\ , \sup_{0\le t\le n-\tau(n)}\left[\frac{\|u(t)\|_{L^\infty(\T)}}{
		\|u(t)\|_{L^1(\T)}}\circ\vartheta_{\tau(n)}\right]\ge (\log n)^{\beta} \right\}\\
	&\le \P\left(\left. \sup_{0\le t\le n} \left[\frac{\|u(t)\|_{L^\infty(\T)}}{
		\|u(t)\|_{L^1(\T)}}\circ\vartheta_{\tau(n)}\right] \ge 
		(\log n)^{\beta} \,\ \right|\, \tau(n)<\infty\right),
\end{align*}
where $\vartheta$ is the same shift on the noise that we introduced
earlier in the context of the application of Kingman's theorem, and 
where we have appealed to the elementary inequality, 
\[\P(A\cap B)\le \P(B\mid A),\]
valid for all events $A,B$ in the underlying probability space.
Thanks to the continuity properties of $u$ and the compactness of $\T$,
\[
	\|u(\tau(n))\|_{L^\infty(\T)} =  (\log n)^{\beta/2} \|u(\tau(n))\|_{L^1(\T)}
	\qquad\text{a.s.\ on $\{\tau(n)<\infty\}$.}
\]
Therefore, the strong Markov property of $\{u(t)\}_{t\ge0}$
ensures that
\[
	\P\left\{ \sup_{0\le t\le n}\frac{\|u(t)\|_{L^\infty(\T)}}{
	\|u(t)\|_{L^1(\T)}} \ge (\log n)^{\beta} \right\} \le \sup_{\substack{v_0\in C_{>0}(\T):\\
	\|v_0\|_{L^\infty(\T)} =  (\log n)^{\beta/2}  \|v_0\|_{L^1(\T)}}}
	\P\left\{ \sup_{0\le t\le n}\frac{\|v(t)\|_{L^\infty(\T)}}{
	\|v(t)\|_{L^1(\T)}} \ge (\log n)^{\beta} \right\},
\]
where $v$ solves \eqref{SHE} subject to initial data $v_0$ that is being optimized under
``$\sup_{v_0\in C_{>0}(\T)\cdots}$,'' and driven by the $\vartheta_{\tau(n)}$-shift
of the Brownian sheet $W$, which is itself a Brownian sheet thanks to the strong
Markov property of the latter, viewed as an infinite-dimensional diffusion in its first variable
 [and with respect to the Gaussian filtration $\F(W)$] . 

Now let $v_0\in C_{>0}(\T)$ be an
otherwise arbitrary continuous and strictly positive 
function such that $\|v_0\|_{L^\infty(\T)}= (\log n)^{\beta/2} \|v_0\|_{L^1(\T)}$. Also,
let $v=\{v(t\,,x)\}_{t\ge0,x\in\T}$ denote the solution to the SPDE \eqref{SHE}, driven by
some space-time white noise $\dot{\mathcal{W}}$, and with the 
initial data $v_0$ that we just fixed in our minds. Define
\[
	V(t\,,x) = \frac{v(t\,,x)}{\|v_0\|_{L^1(\T)}}
	\qquad\text{for all $t\ge0$ and $x\in\T$}.
\]
The random field $V=\{V(t\,,x)\}_{t\ge0,x\in\T}$ solves the SPDE,
\begin{equation}\label{V}
	\partial_t V(t\,,x) = \partial^2_x V(t\,,x) + 
	\hat\sigma(V(t\,,x))\dot{\mathcal{W}}(t\,,x)
	\qquad\text{on $(0\,,\infty)\times\T$},
\end{equation}
subject to $V(0\,,x)= V_0(x)$, where $V_0\in C_{>0}(\T)$ solves
$\|V_0\|_{L^\infty(\T)} = (\log n)^{\beta/2}$ and $\|V_0\|_{L^1(\T)}=1$, and
\[
	\hat{\sigma}(z) = \|v_0\|_{L^1(\T)}^{-1}\sigma\left( 
	z\|v_0\|_{L^1(\T)}\right)\qquad\text{for all
	$z\in\R$}.
\]
Since $\lip(\hat\sigma)=\lip(\sigma)$, any result about the solution $u$ [to \eqref{SHE}]
that depends on $\sigma$ only through $\lip(\sigma)$ can be applied to $V$, regardless 
of our choice of $v_0\in C_{>0}(\T)$. In this way we find the following.  
First define $\mathfrak{S}(N) =\mathfrak{S}(N,\lip(\sigma))$ 
to be the class of all predictable space-time
random fields $V=\{V(t\,,x)\}_{t\ge0,x\in\T}$  that
solve the SPDE \eqref{V} when $\lip(\hat{\sigma})=\lip(\sigma)$,
subject to some non-random initial data $V_0\in C_{>0}(\T)$ that satisfies
$\|V_0\|_{L^1(\T)}=1$ and $\|V_0\|_{L^\infty(\T)}\le N$.
Also define
\begin{equation}\label{NT}\begin{split}
	N &= N(n) =  (\log n)^{\beta/2}\quad\text{to simplify the typesetting, and}\\
	\mathcal{T}_V(R) &= \inf\left\{ t>0:\, \|V(t)\|_{L^\infty(\T)} \ge R\|
	V(t)\|_{L^1(\T)} \right\}\quad\text{for all $R\ge0$}
	\qquad
	[\inf\varnothing=\infty].
\end{split}\end{equation}
[Thus, for example, the elementary inequality
$\|f\|_{L^1(\T)}\le 2\|f\|_{L^\infty(\T)}$, valid for all $f\in L^\infty(\T)$,
tells us that $\mathcal{T}_V(R)=0$ when $0\le R\le\frac12$.]

Then we have
\begin{equation}\label{PsupP}\begin{split}
	\P\left\{ \sup_{0\le t\le n}\frac{\|u(t)\|_{L^\infty(\T)}}{
		\|u(t)\|_{L^1(\T)}} \ge (\log n)^{\beta} \right\} 
           &\le \sup_{V\in\mathfrak{S}(N)}
		\P\left\{ \sup_{0\le t\le n}\frac{\|V(t)\|_{L^\infty(\T)}}{
		\|V(t)\|_{L^1(\T)}} \ge (\log n)^{\beta} \right\}\\
           &\le \sup_{V\in\mathfrak{S}(N)}
		\P\left\{ \mathcal{T}_V(N^2) \le  \exp\left(N^{2/\beta}\right) \right\}.
\end{split}\end{equation}

The general theory of stochastic processes tells us that $\mathcal{T}_V(R)$
is a stopping time with respect to the Gaussian filtration 
$\mathcal{F}(\mathcal{W})$ for every $R\ge0$; see Bass \cite{Bass}.

Define
\[
	\mathcal{P}(m\,,t)  = \sup_{V\in\mathfrak{S}(N)}\P\left\{ \mathcal{T}_V(m) \le t \right\}
	\qquad\text{for all $t,m>0$}.
\]

According to \eqref{PsupP}, we may write
\begin{equation}\label{Boop}
	\P\left\{ \sup_{0\le t\le n}\frac{\|u(t)\|_{L^\infty(\T)}}{
	\|u(t)\|_{L^1(\T)}} \ge (\log n)^{\beta} \right\} \le \mathcal{P}\left( N^2, \exp\left(N^{2/\beta}\right)\right)
	\le J_1 + J_2 + J_3,
\end{equation}
where
\begin{align*}
	J_1 & = \mathcal{P}(N^2,N^{-\gamma}),\\
	J_2 & = \sup_{V\in\mathfrak{S}(N)}\P\left\{ \|V(N^{-\gamma})\|_{L^\infty(\T)} \ge N\right\},\\
	J_3 & = \sup_{V\in\mathfrak{S}(N)}\P\left\{ \|V(N^{-\gamma})\|_{L^\infty(\T)} \le N\ ,
		\sup_{N^{-\gamma} \le t\le  \exp\left(N^{2/\beta}\right)}
		\frac{\|V(t)\|_{L^\infty(\T)}}{\|V(t)\|_{L^1(\T)}} \ge N^2\right\}.
\end{align*}
We study the respective behaviors of $J_1$, $J_2$, and $J_3$ next, and in this order.

\begin{align*}
	J_1 & \le\sup_{V\in\mathfrak{S}(N)}\P\left\{ \sup_{0\le t\le N^{-\gamma}} 
		\|V(t)\|_{L^\infty(\T)} \ge 2N\right\}\\
	&\ + \sup_{V\in\mathfrak{S}(N)}\P\left\{ 
		\sup_{0\le t\le N^{-\gamma}} \|V(t)\|_{L^\infty(\T)} < 2N \ ,
		\inf_{0\le t\le N^{-\gamma}} \|V(t)\|_{L^1(\T)} \le \frac2N\right\}\\
	& \le\sup_{V\in\mathfrak{S}(N)}\P\left\{ \sup_{0\le t\le N^{-\gamma}} 
		\|V(t)\|_{L^\infty(\T)} \ge 2N\right\}
		+ \sup_{V\in\mathfrak{S}(N)}
		\P\left\{\inf_{0\le t\le N^{-\gamma}} \|V(t)\|_{L^1(\T)} \le \frac2N\right\},\\
	&= P_1 + P_2,
\end{align*}
notation being clear from context. For any $\gamma \in \left(\frac43\,, 2\right)$,  the second part of Proposition \ref{pr:KPZ1} ensures that
there exists $K = K(\lip(\sigma))>0$ such that 
\[
	P_1 \le K\exp\left(- K^{-1} N^{(3\gamma-4)/2}\right)\qquad\text{uniformly for all $N\ge 1$}.
\]
Moreover, we may apply 
Proposition \ref{pr:valleys}  
in order to see that there exists $L=L(\lip(\sigma))>0$ such that,
for all $N\ge 4$ [so that $2/N \leq 1/2$], 
\[
	P_2 \le \sup_{V\in\mathfrak{S}(N)}\P\left\{\inf_{0\le t\le N^{-\gamma}}
	\|V(t)\|_{L^1(\T)} \le  \frac{2}{N} \right\}
	\le L \exp\left( - L^{-1} N^{(3\gamma-4)/3} \right),
\]
and the very same inequality holds for all $N\ge1$ if we increase the numerical value of $L$ by a little,
as necessary. Combine to see that there exists a constant $K_1=K_1(\lip(\sigma))>0$
such that
\[
	J_1 \le K_1 \exp\left( - K_1^{-1} N^{(3\gamma-4)/3} \right) \qquad\text{for all $N\ge1$}.
\]

The estimation of $J_2$ is slightly simpler. Indeed, the first assertion of Proposition 
\ref{pr:KPZ1}  readily yields the existence of a number
$K_2=K_2(\lip(\sigma))>0$ such that
\[
	J_2 \le K_2\exp\left( - K_2^{-1} N^{(3\gamma-4)/2} \right) \qquad\text{for all $N\ge1$}. 
\]

Finally, we apply the  Markov property at time $N^{-\gamma}$ in order to see that
\[
	J_3 \le \mathcal{P}\left( N^2 ,  \exp\left(N^{2/\beta}\right) - N^{-\gamma}\right).
\]
Combine the preceding bounds for $J_1,J_2,J_3$ and apply \eqref{Boop}
in order to find that 
\[
	\mathcal{P}\left( N^2, \exp\left(N^{2/\beta}\right)\right) \le L_0\exp\left( - L_0^{-1} N^{(3\gamma-4)/3} \right) 
	+\mathcal{P}\left( N^2, \exp\left(N^{2/\beta}\right)-N^{-\gamma}\right)\qquad\text{for all $N\ge1$},
\]
where $L_0=\max(K_1,K_2)=L_0(\lip(\sigma))>0$. Because $L_0$ does not depend on $N$, and $N$
is an arbitrary real variable $\ge1$, we may iterate the above. Thus, for example
\begin{align*}
	\mathcal{P}\left( N^2, \exp\left(N^{2/\beta}\right)\right)  &\le 2L_0\exp\left( - L_0^{-1} N^{(3\gamma-4)/3} \right) 
		+\mathcal{P}\left( N^2,  \exp\left(N^{2/\beta}\right)-2N^{-\gamma}\right)\\
	&\le 3L_0\exp\left( - L_0^{-1} N^{(3\gamma-4)/3} \right) 
		+\mathcal{P}\left( N^2, \exp\left(N^{2/\beta}\right)-3N^{-\gamma}\right)\\
	&\ \ \vdots\\
	&\le \ell_N L_0\exp\left( - L_0^{-1} N^{(3\gamma-4)/3} \right) 
		+\mathcal{P}\left( N^2,  \exp\left(N^{2/\beta}\right) -\ell_N N^{-\gamma}\right),
\end{align*}
where $\ell_N = \lfloor N^{\gamma} \exp\left( N^{2/\beta}\right) \rfloor $. Since $\beta>3$,  we can  choose $\gamma \in \left(\frac43\,, 2\right)$  such  that 
\begin{equation}\label{eq:exponent} \frac{3\gamma - 4}{3} > \frac{2}{\beta}
\quad \Leftrightarrow \quad \frac{\beta(3\gamma - 4)}{6}>1.
\end{equation}
Since $t\mapsto\mathcal{P}(N^2,t)$
is monotone, it follows the preceding and \eqref{Boop} that
\begin{equation}\label{PPP}\begin{split}
	\P\left\{ \sup_{0\le t\le n}\frac{\|u(t)\|_{L^\infty(\T)}}{
		\|u(t)\|_{L^1(\T)}} \ge (\log n)^{\beta} \right\} &\le L_0 N^{\gamma}
		\exp\left( - L_0^{-1} N^{(3\gamma-4)/3} +N^{2/\beta} \right) 
		+  \mathcal{P}\left( N^2,N^{-\gamma}\right)\\
	&=L_0  N^{\gamma}  \exp\left( - L_0^{-1} N^{(3\gamma-4)/3} + N^{2/\beta} \right)  + J_1,\\
	&\le L_1\left( N^{\gamma}+1\right)\exp\left( - L_1^{-1} N^{(3\gamma-4)/3} \right),
\end{split}\end{equation}
thanks to a final appeal to our estimate for $J_1$, and for a suitably large choice of
$L_1=L_1(\lip(\sigma))>0$. A final appeal to \eqref{NT} allows
to change variables back from $N$ to $n$ and we obtain 
\begin{equation} \label{eq:Borel-Cantelli-sup-inf}
	\sum_{n=1}^\infty \P\left\{ \sup_{0\le t\le n}\frac{\|u(t)\|_{L^\infty(\T)}}{
	\|u(t)\|_{L^1(\T)}} \ge (\log n)^{\beta} \right\}\leq  L_1 \sum_{n=1}^\infty (1+\log n)^{\gamma\beta/2}\,  \exp\left( - L_1^{-1} (\log n)^{\beta(3\gamma-4)/6} \right)   <   \infty,
\end{equation}
for every $\beta>3$. Note the sum above converges, thanks to \eqref{eq:exponent}. Lemma \ref{lem:WLOG} now implies Theorem \ref{th:S/M}.\qed

\section{Proof of Theorem \ref{th:inf}} \label{sec:proof_main}

The proof of Theorem \ref{th:inf} is divided into a series of natural steps that we
separate as individual subsections below. Throughout, we will appeal also to the
following fact from elementary probability theory.

\begin{lemma}\label{lem:cond:prob}
	For all integers $n\ge2$ and
	events $E_1,\ldots,E_n$ in the underlying probability space,
	\[
		\P(E_n) \le \sum_{j=1}^{n}\P\left(
		E_j \mid E_0^c\cap \cdots\cap E_{j-1}^c \right),
		\qquad\text{where $E_0=\varnothing$.}
	\]
\end{lemma}

\begin{proof}
	Define $p_1 = \P(E_1) $ and set $p_j = \P(E_j\mid E_1^c\cap\cdots\cap E_{j-1}^c)$
	for $j=2,\ldots,n$. Then,
	\[
		\P(E_n^c) \ge \P\left(E_1^c\cap\cdots\cap E_n^c\right) = \prod_{j=1}^n (1-p_j).
	\]
	The result follows from this and the elementary inequality 
	$\prod_{j=1}^n(1-p_j) \ge 1- \sum_{j=1}^n p_j$, valid for every $0\le p_1,\ldots,p_n\le1$
	as can be checked directly by induction.
\end{proof}

\subsection{The influence of the heat kernel}
Let us start with a small technical result about the smoothing action of the heat
semigroup on the torus. It is well known that, in arbitrary positive time, the heat semigroup 
maps an integrable function to a smooth one. The following provides a quantitative bound
for that smoothness, where ``smoothness'' is  interpreted here in terms of optimal Lipschitz 
constants.

\begin{lemma}\label{lem:delta:p}
	If $t>0$ and $f\in L^1(\T)$, then $P_tf$ is Lipschitz continuous, and
	\[
		\lip(P_tf) \le \frac{7}{\sqrt{t}}\left( 1 + 
		\frac{1}{\sqrt t}\right)\|f\|_{L^1(\T)}.
	\]
\end{lemma}

\begin{proof}
	We can differentiate the series representation \eqref{p} term by term in order to see that
	\[
		|\partial_x  p_t(x\,,y)| \le  \sum_{n=-\infty}^\infty \frac{|x-y+2n|}{t\sqrt{4\pi t}}
		\exp \left( -\frac{|x-y+2n|^2}{4t}   \right)
		\qquad\text{pointwise}.
	\]
	If $|n|\ge2$, then certainly $|x-y|\le 2\le |n|$,
	whence also $|x-y+2n|\ge 2|n|-|x-y|\ge|n|$
	for all $x,y\in\T$. Since $|x-y+2n|\le2(1+|n|)\le3|n|$ as well, it follows that
	\[
		\sum_{n\in\Z:|n|\ge2}\frac{|x-y+2n|}{t\sqrt{4\pi t}}
		\exp \left( -\frac{|x-y+2n|^2}{4t}   \right)
		\le\frac{3}{t\sqrt{4\pi t}}\sum_{n\in\Z:|n|\ge2}|n|\e^{-n^2/(4t)}.
	\]
	Now,
	\begin{align*}
		\sum_{n=1}^\infty n\e^{-n^2/(4t)}
			&=\sum_{n=1}^\infty \int_{n-1}^n
			n\e^{-n^2/(4t)}\,\d y
			\le\sum_{n=1}^\infty \int_{n-1}^n
			(1+y)\e^{-y^2/(4t)}\,\d y\\
		&= \int_0^\infty\e^{-y^2/(4t)}\,\d y + 
			\int_0^\infty y\e^{-y^2/(4t)}\,\d y
			= \sqrt{\pi t} + 2t < \sqrt{4\pi t}(1+\sqrt t).
	\end{align*}
	Therefore, by symmetry,
	\[
		\sum_{n\in\Z:|n|\ge2}\frac{|x-y+2n|}{t\sqrt{4\pi t}}
		\exp \left( -\frac{|x-y+2n|^2}{4t}   \right)
		\le\frac{6}{t\sqrt{4\pi t}} \sqrt{4\pi t}(1+\sqrt t)
		= \frac{6}{\sqrt{t}} \left( 1 + \frac{1}{\sqrt t}\right).
	\]
	When $|n|\le1$, we use the fact that $|z|\exp(-z^2/(4t))\le\sqrt{2t/\e}$
	for all $z\in\R$ and $t>0$  in order to be able to say that
	\[
		\sum_{n\in\Z:|n|\le1}\frac{|x-y+2n|}{t\sqrt{4\pi t}}
		\exp \left( -\frac{|x-y+2n|^2}{4t}   \right)
		\le \frac{3}{t\sqrt{2\pi\e}} < \frac{1}{\sqrt{t}} \left( 1 + \frac{1}{\sqrt t}\right).
	\]
	Combine these bounds in order to find that
	\[
		\lip(p_t(\cdot\,,y))  = \sup_{x\in\T}|\partial_x  p_t(y\,,x)| 
		\le  \frac{7}{\sqrt{t}}\left( 1 + \frac{1}{\sqrt t}\right)
		\qquad\text{for all $t>0$ and $y\in\T$}.
	\]
	This bound does the job since
	$| (P_tf)(x) - (P_tf)(z) |\le
	\int_\T \left| p_t(x\,,y)  -  p_t(z\,,y)\right| |f(y)|\,\d y$.
\end{proof}

Let 
\begin{equation}\label{ball}
	B(x\,,r)=\{x+r z\in\T:\, z\in\T\}
\end{equation}
denote the closed
$r$-ball about $x\in\T$ for all $r\ge0$. 
Because of  \eqref{p},
\begin{equation}\label{A} 
	\int_{B(a,c\sqrt t)} p_t(x\,,y) \,\d y
        = \frac{1}{\sqrt{4\pi t}}\sum_{n=-\infty}^\infty
	\int_{a-c\sqrt t}^{a+c\sqrt t} \exp\left( -\frac{(y-x+2n)^2}{4t}\right)\d y,
\end{equation}
for every $t>0$, $c\in[0\,,1/\sqrt t]$, and $a,x\in\T$.  For the right-hand side it might help
to recall the convention that we are
identifying points in $\T$ with those in $[-1\,,1]$. Still,
to be completely clear, the left-hand side
is an integral  on the abelian group $\T$ against the Haar measure $[\d y]$, 
and the right-hand side is a sum of
integrals over the real line against Lebesgue's measure [also denoted by $\d y$].

\begin{lemma}\label{lem:p:LB}
	 There exists a constant $A \in (0\,,1)$ such that  for every $c\geq 1$, $t>0$, and $a,x\in\T$,
	\[
		\int_{B(a,c\sqrt t)} p_t(x\,,y)\,\d y \ge A
		\,\1_{B(a,(c+1)\sqrt t)}(x).
	\]
\end{lemma}

\begin{proof}
	Because the integral in question is manifestly $\ge0$, we may (and will)
	consider only $x\in B(a\,,(c+1)\sqrt t)$.
	In \eqref{A}, we drop all summands except the one that corresponds to $n=0$  to see that   
	\begin{align*} 
		\int_{B(a,c\sqrt t)} p_t(x\,,y)\,\d y &\geq  \frac{1}{\sqrt{4\pi t}} 
		\int_{a-c\sqrt t}^{a+c\sqrt t} \exp\left( -\frac{(y-x)^2}{4t}\right)\d y\\
		&=\P\left\{ \frac{a-x}{\sqrt{2t}} - \frac{c}{\sqrt{2}} \leq Z \leq 
			\frac{a-x}{\sqrt{2t}} +\frac{c}{\sqrt{2}} \right\},
	\end{align*}
	where $Z$ is distributed according to the standard normal distribution on $\R$. 
	Since $x$ is within $a\pm (c+1)\sqrt t$ and $c\geq 1$, it follows that
	\begin{equation}\label{def:A}
		\int_{B(a,c\sqrt t)} p_t(x\,,y)\,\d y
		\geq  \P\left\{ \frac{1}{\sqrt{2}} \leq Z \leq \frac{2c+1}{\sqrt{2}} \right\}
		\geq \P\left\{ \frac{1}{\sqrt{2}} \leq Z \leq \frac{3}{\sqrt{2}} \right\} =:A.
	\end{equation}
	This completes the proof.
\end{proof}

\subsection{The influence of the noise}
Lemmas \ref{lem:delta:p} and \ref{lem:p:LB} concern the heat kernel/semigroup.
We now begin studying the noise $\dot{W}$. In order to do that, let us
choose and fix some numbers
\begin{equation}\label{alpha}
 \beta>3 \qquad \text{and} \qquad  \alpha>\frac{4\beta}{3} + 1 > 5 
 \end{equation}
and define successively,
\begin{equation}\label{def:t_k}
	t_1=1, \qquad t_2=2,\qquad\text{and}\qquad
	t_{k+1} = t_k + \frac{1}{(\log k)^\alpha} \quad \text{for $k\geq 2$}.
\end{equation}
 Note that $t_k\ge 1$ for all $k\in\N$, and  $k(\log k)^{-\alpha} \leq t_k \leq k$ for all $k\geq 2$. 
In particular, we have 
\begin{equation}\label{t_k}
	\frac{\log k}{2}  \le \log  t_k \le \log k 
	\qquad\text{for all $k\geq 1$}.
\end{equation} 

With the above sequence of $t_k$'s in place, 
we consider the events $\bcA_{1,1},\bcA_{2,1},\ldots$
defined via the following:
\begin{equation*} 
	\bcA_{k,1}  =\left\{ \omega\in\Omega:\,
	\sup_{t\in(0,t_k]}\frac{\|u(t)\|_{L^\infty(\T)}(\omega)}{\|u(t)\|_{L^1(\T)}(\omega)}
	<  (\log t_k)^\beta  \right\}
\end{equation*}
These are large sets in the underlying probability space when $k\gg1$. In fact, we have the following.
\begin{lemma}\label{lem:P(A1)}
	 There exist numbers $c=c(\alpha\,,\beta\,,\lip(\sigma))>0$ and 
	 $\eta_1=\eta_1(\beta)>1$ such that 
	\[
		\P(\bcA_{k,1}^c) \lesssim \exp\left( -c  (\log k)^{\eta_1}  \right)
		\quad\text{uniformly for all $k\in\N$.}
	\] 
\end{lemma}

\begin{proof}
	When used in conjunction, \eqref{PPP} and \eqref{eq:Borel-Cantelli-sup-inf} imply that there exist  
	constants $L_1=L_1(\lip(\sigma))>0$ and  
	$\gamma_1=\gamma_1(\beta)\in \left( \frac43\,, 2\right)$  such that 
	$(3\gamma_1-4)\beta/6>1$ and
	\[
		\P(\bcA_{k,1}^c)
		\le L_1 (1+\log t_k)^{\gamma_1 \beta/2} 
		\exp\left( - L_1^{-1} (\log t_k)^{\beta(3\gamma_1-4)/6} \right)  ,
	\]
	uniformly for all $k\in\N$. We can choose $\eta_1=\beta(3\gamma_1 - 4)/6$,
	and enlarge $L_1=L_1(\alpha\,,\beta\,,\lip(\sigma))$
	if needed, in order to deduce the lemma from this, \eqref{alpha}, and  \eqref{t_k}. 
\end{proof}

Next, we define a sequence $v_1,v_2,\ldots$ of space-time random fields as follows:
\begin{equation}\label{v_k}
	v_k(t) = \frac{u(t)}{\|u(t_{k-1})\|_{L^1(\T)}}
	\qquad \text{for every $t\ge t_{k-1}$,}
\end{equation}
Consider the events,
\begin{equation}\label{A:k,2}
	\bcA_{k,2} =   \left\{\omega\in\Omega:\,
	\frac{1}{2}\leq \|v_k(t)\|_{L^1(\T)}(\omega) \leq 2
	\quad
	\text{for every $t\in[t_{k-1}\,,t_k]$}\right\},
\end{equation}
and recall that $\eta_1=\eta_1(\beta)>0$ is the number given in Lemma \ref{lem:P(A1)}.  Then we have the following result.

\begin{lemma}\label{lem:P(A2)}
	There exist  numbers $c=c(\alpha\,, \beta\, , \lip(\sigma))>0$ and 
	$\eta_2=\eta_2(\alpha\,, \beta) \in (1\,, \eta_1]$ such that
	\[
		\P(\bcA_{k,2}^c ) \lesssim
		\exp\left( -c  (\log k)^{\eta_2}   \right)
		\qquad\text{uniformly for all $k\in\N$.}
	\]
\end{lemma}

Lemma \ref{lem:P(A2)} is a quantitative way to say that, 
with very high probability, the total-mass process
$\int_\T u(t\,,y)\,\d y = \|u(t)\|_{L^1(\T)}$ does not move much
for $t\in(t_{k-1}\,,t_k)$.

\begin{proof}
	We shall consider integers $k\ge2$. The $k=1$ case is easy to include [by adjusting constants
	only] and will not be discussed.
	
	Let us observe that $\|v_k(t_{k-1})\|_{L^1(\T)}=1$  almost surely.
	We emphasize that, by virtue of definition, 
	$\|v_k(t_{k-1})\|_{L^\infty(\T)}\le  (\log t_k)^\beta $ almost surely on $\bcA_{k-1,1}$. 
	
	According to \eqref{alpha}, we can choose  $\gamma\in(\frac43\,,2)$ that  satisfies
	\begin{equation}\label{cond:gamma}
		\gamma\leq \frac{\alpha}{\beta} \quad \text{and} \quad \frac{(3\gamma -4)\beta}{3}>1.
	\end{equation}
	We apply Proposition \ref{pr:valleys}, with $N= (\log t_k)^\beta $,
	together with the Markov property of the
	solution to \eqref{SHE} in order to find that 
	there exists $L=L(\gamma\,,\lip(\sigma))>0$ such that,
	a.s.\ on $\bcA_{k-1,1}$,
	 \begin{align*}
		&\P\left(\left. \|v_k(t)\|_{L^1(\T)} \not\in (1/2\,,2)
			\quad\text{for some }t\in\left[t_{k-1}\,,t_{k-1}+
			(\log t_k)^{-\gamma\beta}\right]\ \,\right|\, \F_{t_{k-1}}\right)\\
		&\hskip.5in=\P\left(\left. \|v_k(t)\|_{L^1(\T)} \not\in (1/2\,,2)
			\quad\text{for some }t\in\left[t_{k-1}\,,t_{k-1}+
			N^{-\gamma}\right]\ \,\right|\, \F_{t_{k-1}}\right)\\
		&\hskip1cm\le
			L\exp\left( -\frac{N^{(3\gamma-4)/3}}{L} \right)
			= L\exp\left( -\frac{ (\log t_k)^{(3\gamma-4)\beta/3}}{L} \right)
			\le L\exp\left( -\frac{(\log k)^{(3\gamma-4)\beta/3}}{%
			2^{(3\gamma-4)\beta/3}L} \right),
	\end{align*} 
	uniformly for all integers $k\ge2$; see \eqref{t_k} for the last
	inequality. Because of \eqref{t_k} and \eqref{cond:gamma},
	\[
	 t_{k-1} +(\log  t_k)^{-\gamma\beta} \ge t_{k-1} + 
		(\log k)^{-\gamma\beta} \geq  t_{k-1} + (\log k)^{-\alpha} = t_k ,
	\]
	for all  $k\ge2$. In other words,
	$[t_{k-1}\,,t_{k-1}+ N^{-\gamma}] \supset [t_{k-1}\,, t_k]$
	for all  $k\ge2$. 
	Therefore, a second appeal to
	\eqref{t_k} yields
	\[
		\P\left( \bcA_{k,2}^c\mid\F_{t_{k-1}}\right)\\
		\le L\exp\left( -\frac{(\log k)^{(3\gamma-4)\beta/3}}{%
			2^{(3\gamma-4)\beta/3}L} \right),
	\]
	a.s.\ on $\bcA_{k-1,1}$, for all but a finite number of integers $k\ge1$.
	In particular, we appeal to the fact that $\bcA_{k-1,1}$ is $\F_{t_{k-1}}$-measurable
	in order to deduce the following:
	\begin{equation}\label{P(A2)}
		\P\left( \bcA_{k,2}^c\mid\bcA_{k-1,1}\right)\\
		\le L\exp\left( -\frac{(\log k)^{(3\gamma-4)\beta/3}}{%
			2^{(3\gamma-4)\beta/3}L} \right)\qquad\text{for all $k\in\N$}.
	\end{equation}
	According to Lemma \ref{lem:cond:prob},
	$\P(\bcA_{k,2}^c)\le \P(\bcA_{k-1,1}^c) + \P( \bcA_{k,2}^c\mid\bcA_{k-1,1})$.
	We now choose $\eta_2=\eta_1\wedge \tfrac13(3\gamma -4)\beta$ in order
	to deduce  the result from the preceding
	and from Lemma \ref{lem:P(A1)}.
\end{proof}

Next, let us consider for every $q,L>0$ and $k\in\N$ the event
\begin{equation*}
	\bcA_{k,3} = \bcA_{k, 3}(q\,,L) = \left\{\omega\in\Omega:
	\sup_{\substack{x,y\in\T\\ x\neq y}} 
	\frac{ |v_k(t_k\,,x)-v_k(t_k\,,y)|}{|x-y|^q}
        (\omega) \le  L (\log k)^\alpha \right\}.
\end{equation*}
These are large-probability events, provided that $q$ and $L$ are chosen appropriately.
In fact, we have the following.

\begin{lemma}\label{lem:P(A3)}
	For every $q\in(0\,,1/2)$, there exist numbers
	$L=L(\alpha, \beta\,,q\,,\lip(\sigma))>0$,
	$c=c(\alpha\,, \beta\,,q\,,\lip(\sigma))>0$,  and $\eta_3=\eta_3(\alpha\,, \beta)\in (1\,,  \eta_2]$
	such that, uniformly for all  $k\in\N$,
	\[
		\P(\bcA_{k,3}^c) \lesssim
		\exp\left( - c\,   (\log k)^{\eta_3}\right),
	\]
	where $\eta_2>0$ is the number given in Lemma \ref{lem:P(A2)}.
\end{lemma}

\begin{proof}
	It suffices to prove the lemma for $k\in\N\cap[3\,,\infty)$.
	
	We may apply the Markov property at time $t_{k-1}$ to see that, started from time $t_{k-1}$,
	$v_k$ solves the \eqref{SHE} with a suitable rescaled version of $\sigma$, 
	and with initial profile $v_k(t_{k-1})$. More
	precisely, we may write $v_k$ in the following mild form:
	\begin{equation}\label{vkII:g}
		v_k(t)= P_{t-t_{k-1}} [v_k(t_{k-1})] + 
		\mathcal{I}_{v,k}(t)\qquad\text{for all $t>t_{k-1}$},
	\end{equation}
	where $p$ and $P$ denote respectively the heat kernel
	and semigroup [see \eqref{p} and \eqref{P}], and
	\begin{equation}\label{eq:noise}
		\mathcal{I}_{v,k}(t\,,x)
		= \int_{(t_{k-1},t)\times\T}p_{t-s}(x\,,z)
		\sigma_k \left( v_k(s\,,z) \right)W(\d s\,\d z),
	\end{equation}
	with
	\begin{equation}\label{sigma_k}
		\sigma_k (w) =  \frac{\sigma\left( \|u(t_{k-1})\|_{L^1(\T)} w  \right)}{%
		\|u(t_{k-1})\|_{L^1(\T)}}\qquad
		\text{for all $w\in\R$}.
	\end{equation}
	In order to simplify the typography, we will write
	\[
		E_k = \E(\,\cdots\mid\F_{t_{k-1}})\qquad[k\in\N].
	\]
	The random function $\sigma_k$ is $\F_{t_{k-1}}$-measurable,
	and has the same optimal Lipschitz constant as $\sigma$. Because 
	$t_k-t_{k-1}=(\log k)^{-\alpha}\le 1$ for all $k\in\N\cap[3\,,\infty)$, we
	may apply Lemma \ref{lem:tp1} conditionally using the Markov property 
	in order to see that there exists $c=c(\lip(\sigma_k))=c(\lip(\sigma))>0$ such that, 
	uniformly for every $k\in\N\cap[3\,,\infty)$, $\nu\in[2\,,\infty)$,
	$s,t\in(t_{k-1}\,,t_k]$, and $x,y\in\T$,
	\begin{align*}
		E_k\left( \left| \mathcal{I}_{v,k}(t\,,x) - \mathcal{I}_{v,k}(s\,,y) \right|^\nu\right)
		\le (c\nu)^{\nu/2} \|v_k(t_{k-1})\|_{L^\infty(\T)}^\nu
		\exp\left(2c\nu^3  (\log k)^{-\alpha}  \right)\left( \sqrt{|t-s|} + |x-y| \right)^{\nu/2},
	\end{align*}
	almost surely. Therefore, a suitable formulation of the Kolmogorov continuity theorem
	\cite[Appendix C]{cbms} yields, for every $q\in(0\,,1/2)$
	a number $c_1=c_1(q\,,\lip(\sigma))>0$ such that, 
	uniformly for every $k\in\N\cap[3\,,\infty)$, $\nu\in[2\,,\infty)$,
	$s,y\in(t_{k-1}\,,t_k]$, and $x,y\in\T$,
	\begin{align*}
		E_k\left( \sup_{\substack{x,y\in\T\\ x\neq y}}
		\frac{\left| \mathcal{I}_{v,k}(t_k\,,x) - \mathcal{I}_{v,k}(t_k\,,y) \right|^\nu}{%
		|x - y|^{q\nu}}\right)
		\le (c_1\nu)^{\nu/2} \|v_k(t_{k-1})\|_{L^\infty(\T)}^\nu
		\exp\left(2c\nu^3   (\log k)^{-\alpha} \right),
	\end{align*}
	almost surely. Since $\lip(P_tf)\lesssim t^{-1}\|f\|_{L^1(\T)}$ for all $t\in(0\,,2]$
	and $f\in L^1(\T)$ [see Lemma \ref{lem:delta:p}], we can deduce from
	\eqref{vkII:g} that for every $q\in(0\,,1/2)$ there exists
	a number $c_2=c_2(q\,,\lip(\sigma))>0$ such that, 
	uniformly for every $k\in\N\cap[3\,,\infty)$, $\nu\in[2\,,\infty)$,
	$s, t\in(t_{k-1}\,,t_k]$, and $x,y\in\T$,
	\begin{align*}
		&E_k\left( \sup_{\substack{x,y\in\T\\ x\neq y}}
			\frac{\left| v_k(t_k\,,x) - v_k(t_k\,,y) \right|^\nu}{%
			|x - y|^{q\nu}}\right)  \\
		&\hskip1in\le (c_2\nu)^{\nu/2} \|v_k(t_{k-1})\|_{L^\infty(\T)}^\nu
			\exp\left(2c\nu^3   (\log k)^{-\alpha} \right) + \frac{c_2^\nu
			\|v_k(t_{k-1})\|_{L^1(\T)}^\nu}{(t_k-t_{k-1})^\nu}\\
		&\hskip1in\le c_2^\nu\left[
			\nu^{\nu/2}\|v_k(t_{k-1})\|_{L^\infty(\T)}^\nu
			\exp\left(2c\nu^3   (\log k)^{-\alpha}  \right) +  (\log k)^{\alpha\nu}  \|v_k(t_{k-1})\|_{L^1(\T)}^\nu\right],
	\end{align*}
	almost surely.  We have already observed that
	$\| v_k(t_{k-1})\|_{L^1(\T)}=1$ [see also \eqref{v_k}] and
\begin{equation*}
   \|v_k(t_{k-1})\|_{L^\infty(\T)}\le  (\log t_k)^\beta \leq (\log k)^\beta
\end{equation*}
	a.s.\ on $\bcA_{k-1,1}$ [see \eqref{t_k} for the last inequality]. Therefore, we get that 
	\begin{align*}
		\E\left(\left. \sup_{\substack{x,y\in\T\\ x\neq y}}
		\frac{\left| v_k(t_k\,,x) - v_k(t_k\,,y) \right|^\nu}{%
		|x - y|^{q\nu}}\ \,\right|\, \F_{t_{k-1}}\right)
		\le c_2^\nu \left[\sqrt\nu\, (\log k)^\beta 
		\exp\left(2c\nu^2 (\log k)^{-\alpha}\right) +  (\log k)^{\alpha} \right]^\nu,
	\end{align*} 
	a.s.\  on $\bcA_{k-1,1}\cap\bcA_{k,2}$.  Define
	\begin{equation*}
		\nu:=(\log k)^{\alpha/2} 
	\end{equation*} 
	and apply the preceding with this $\nu$ to see that 
	\begin{align*}
		\left[\E\left(\left. \sup_{\substack{x,y\in\T\\ x\neq y}}
			\frac{\left| v_k(t_k\,,x) - v_k(t_k\,,y) \right|^\nu}{%
			|x - y|^{q\nu}}\ \,\right|\, \F_{t_{k-1}}\right)\right]^{1/\nu}
			& \le c_4\left[   (\log k)^{\beta+\alpha/4}+ (\log k)^\alpha   \right]\\
		&\le  2c_4  (\log k)^\alpha \hskip1in[\text{see \eqref{alpha}}]  ;
	\end{align*}
	a.s.\  on $\bcA_{k-1,1}\cap\bcA_{k,2}$, where $c_4=c_4(\alpha\,,q\,,\lip(\sigma))>0$ 
	is a non-random number that does not depend on $k\in\N$. 
	Therefore, an application of Chebyshev's inequality yields
	\[
		\P\left( \left. \sup_{\substack{x,y\in\T\\ x\neq y}}
		\frac{\left| v_k(t_k\,,x) - v_k(t_k\,,y) \right|}{%
		|x - y|^q} \ge 2\e c_4 (\log k)^\alpha \ \,\right|\, \F_{t_{k-1}}\right)
		\le \e^{-\nu},
	\]
	a.s.\  on $\bcA_{k-1,1}\cap\bcA_{k,2}$. Because $\bcA_{k-1,1}$ is $\F_{t_{k-1}}$-measurable,
	we deduce from the above and \eqref{alpha} that
	\begin{equation}\label{P(A3)}
		\P\left( \bcA_{k,3}^c\mid\bcA_{k-1,1}\cap\bcA_{k,2}\right) \le
		 \exp\left(-(\log k)^{\alpha/2}\right)   \qquad\text{for all $k\in\N$}.
	\end{equation}
	Lemma \ref{lem:cond:prob} ensures that
	\[
		\P(\bcA_{k,3}^c)\le \P(\bcA_{k-1,1}^c) +
		\P(\bcA_{k,2}^c \mid \bcA_{k-1,1}) +
		\P(\bcA_{k,3}^c\mid\bcA_{k-1,1}\cap\bcA_{k,2}).
	\]
	Combine this with Lemma \ref{lem:P(A1)}, \eqref{P(A2)}, and \eqref{P(A3)}, and let 
	\begin{equation}\label{eta3}
		\eta_3:=\min\left\{\eta_1\,, \eta_2\,,\frac{\alpha}{2}\right\}> 1
	\end{equation}  in order to conclude the proof.
\end{proof}

\subsection{Conclusion of the proof of Theorem \ref{th:inf}}
Armed with the technical results of the previous two subsections, and with Theorem 
\ref{th:S/M}, equivalently \eqref{eq:Lpq}, 
we now work toward completing the proof of Theorem  \ref{th:inf}. Because
we have already proved \eqref{eq:Lpq} which states
that the $L^1(\T)$- and $L^\infty(\T)$-norms are
close up to logarithmic errors, it remains to prove that with probability one
the infimum does
not stray away from the $L^1(\T)$-norm by more than a logarithmic term
as time tends to infinity. Therefore, we can see that
the following clearly is a big step in the right direction.
The notation is the same as that of the earlier portions of this section.
In particular, $t_1<t_2<\cdots$ is the sequence that was defined in
\eqref{def:t_k}; see also \eqref{t_k}.

\begin{lemma} \label{lem:v_k:LB}
	Choose and fix a real number $\kappa >2\alpha$. 	
	Then, there exists a non-random sequence $0<s_1<s_2<\cdots$ such that
	$s_k+t_k\le t_{k+1}$ for all $k\in\N$ sufficiently large, and with probability one,
	\[
		\inf_{x\in\T} u(s_k+t_k\,,x)  > \exp\left( -7 (\log t_k)^{\kappa} \right)   \|u(t_{k-1})\|_{L^1(\T)}
		 \qquad\text{for all but a finite number of
		 $k\in\N$}.
	\] 
\end{lemma}

\begin{proof}
	As before, let us identify $\T$ with the  interval $[-1\,,1]$ and 
	with that in mind define
	\[	
		x_k = \inf\left\{ x\in\T:\, v_k(t_k\,,x) \ge \tfrac14\right\},
	\]
	with $\inf\varnothing=2$ (any point that is not in $[-1\,,1]$ will do). Then, $x_k$ is
	an $\F_{t_k}$-measurable random variable that takes its values in $[-1\,,1]\cup\{2\}$.
	Because $\|v_k(t_k)\|_{L^1(\T)}\ge\frac12$ a.s.\ on $\bcA_{k,2}$
	[see \eqref{A:k,2}], we can conclude that
	\begin{equation}\label{x_k}
		x_k\in\T\quad\text{a.s.\ on $\bcA_{k,2}$}.
	\end{equation}
	Choose and fix  $q\in(0\,,1/2)$ and $L>0$ according to the statement
	of Lemma \ref{lem:P(A3)} and let $\kappa$ be any number such that 
	\begin{equation}\label{eq:lambda}
		\kappa > \frac{\alpha}{q}.
	\end{equation}
	Since we can choose $\alpha$ to be any number in $(4\,,\infty)$ -- see \eqref{alpha} --
	and  because $q\in(0\,,1/2)$ is otherwise arbitrary, 
	we can choose $\kappa$ to be any number that satisfies
	\begin{equation} \label{eq:kappa-gt-8}
		\kappa>10.
	\end{equation}
	We also  define
	\[
        	r_k = (\log t_k)^{-\kappa}\qquad[k\in\N].
	\]
	Then, recall \eqref{ball} and observe that for every $k\in\N$,
	\begin{align*}
		\inf_{y\in B(x_k,r_k)} v_k(t_k\,,x_k) &\ge 
			v_k(t_k\,,x_k) - \sup_{y\in B(x_k,r_k)}\left| v_k(t_k\,,y) - v_k(t_k\,,x_k)\right|\\
		&\ge \frac14 - L r_k^q (\log k)^\alpha   \qquad\text{a.s.\ on 
			$\bcA_{k,2}\cap\bcA_{k,3}$}\\
		&= \frac14 - 2^{\kappa}L (\log k)^{-(\kappa q -\alpha)};
	\end{align*}
	see \eqref{t_k}. We set 
	\[
	 	k_\alpha = \left\lceil \exp\left( (2^{\kappa+3}L)^{1/(\kappa q-\alpha)} \right)  \right\rceil,
	\] 
	where $\lceil y\rceil$ denotes the smallest integer to the right of $y\ge0$, as is customary.
	In particular, we may observe that $k_\alpha\in\N$ is deterministic and
	\begin{equation}\label{interval}
		\inf_{y\in B(x_k,r_k)} v_k(t_k\,,x_k) \ge \frac{1}{8}
		\qquad\text{a.s.\ on 
		$\bcA_{k,2}\cap\bcA_{k,3}$ for all integers $k\ge k_\alpha$.}
	\end{equation}
	As in \eqref{vkII:g},
	\begin{equation}\label{vkII:gg}
		v_k(t + t_k ) = P_t [v_k(t_k)] + 
		\mathcal{I}_{v,k+1}(t + t_k)\qquad\text{for all $t>0$},
	\end{equation}
	where $p$ and $P$ denote respectively the heat kernel
	and semigroup [see \eqref{p} and \eqref{P}], and $\mathcal{I}_{v,k+1}$ and $\sigma_{k+1}$
	were respectively defined in \eqref{eq:noise} and \eqref{sigma_k}. We study the two terms 
	on the right-hand side of \eqref{vkII:gg} separately.
	
	Thanks to \eqref{interval}, the deterministic quantity on the right-hand side of
	\eqref{vkII:gg} satisfies the following:
	\[
		P_t[v_k(t_k)] \ge \frac{1}{8}\int_{B(x_k,r_k)} p_t(\cdot\,,y)\,\d y
		\qquad\text{a.s.\ on $\bcA_{k,2}\cap\bcA_{k,3}$}.
	\]
	Therefore, Lemma \ref{lem:p:LB} [with $c=1$] ensures that, for all $k\in\N$,
	\begin{equation}\label{pre:t:t_k}
		P_t[v_k(t_k)] \ge \frac{A}{8}\,
		\1_{B(x_k,2\sqrt{t})} = \e^{-\chi}\, \1_{B(x_k,2\sqrt{t})}
		\qquad\text{a.s.\ on $\bcA_{k,2}\cap\bcA_{k,3}$},
	\end{equation}
	provided additionally that $0 < 4 t< r_k^2 =(\log t_k)^{-2\kappa}.$
	The following choice will therefore do:
	\begin{equation}\label{t:t_k}
		t = \tfrac18 (\log t_k)^{-2\kappa }.
	\end{equation}
	Henceforth, the symbol ``$t$'' is reserved for the particular choice in \eqref{t:t_k}.
	To be sure, we mention also that \eqref{pre:t:t_k} defines the number
	$\chi$ independently of all other parameters; in fact, in light of \eqref{def:A},
	we have
	\[
		\chi = \log 8 - \log A=\log 8 - \P\left\{
		\frac{1}{\sqrt{2}} \leq Z \leq \frac{3}{\sqrt{2}} \right\} \geq \log 8 -1  \approx 1.08.
	\]
	
	Let us turn to the second quantity on the right-hand side of \eqref{vkII:gg}.
	Lemma \ref{lem:tp1} and an application of the Markov property at time $t_k$
	together yield the following: There exists  constant $c = c(\lip(\sigma))>0$
	such that simultaneously for all $\nu\in[2\,,\infty)$, $k\in\N$, $x,y\in\T$,
	and $s_1,s_2\in [0\,,t]$,
	\begin{align*}
		&\E\left(\left. | \mathcal{I}_{v,k+1}(s_1+t_k\,,x) - \mathcal{I}_{v,k+1}(s_2+t_k\,,y)
			|^\nu\ \,\ \right|\, \F_{t_k}\right)^{1/\nu}\\
		&\lesssim \sqrt{\nu}\, \exp(c\nu^2 t ) \|v_k(t_k)\|_{L^\infty(\T)}
			\left[ \sqrt{|s_2-s_1|}+|x-y|\right]^{1/2}\\
		&\le\sqrt\nu\, \exp(c\nu^2 t)  (\log k)^\beta  \left[ \sqrt{|s_2-s_1|}+|x-y|\right]^{1/2}
			\qquad\text{a.s.\ on $\bcA_{k,1}$}.
	\end{align*}
	Moreover, the implied constant is deterministic, as can be seen from inspecting the details
	of the arguments that leads to these bounds.
	Once again, we appeal to this and a suitable form of the Kolmogorov continuity theorem
	in order to conclude that
	\begin{align*}
		&\E\left(\left. \sup_{\substack{s_1,s_2\in(0,t)\\ s_1\neq s_2}}
			\sup_{x\in\T} \frac{| \mathcal{I}_{v,k+1}(s_1+t_k\,,x) 
			- \mathcal{I}_{v,k+1}(s_2+t_k\,,x)
			|^\nu}{|s_2-s_1|^{q\nu/2}}\ \,\ \right|\, \F_{t_k}\right)^{1/\nu}\\
		&\lesssim \sqrt{\nu}\exp\left( c\nu^2t\right)  (\log k)^\beta  
			\qquad\text{a.s.\ on $\bcA_{k,1}$},
	\end{align*}
	where the implied constant is independent of $\nu\in[2\,,\infty)$ and $k\in\N$, as well as
	deterministic. In particular, this yields
	\begin{align*}
		&\hspace{-1cm}\E\left(\left. \adjustlimits\sup_{s\in(0,t)}
			\sup_{x\in\T}| \mathcal{I}_{v,k+1}(s+t_k\,,x) |^\nu
			\ \,\ \right|\, \F_{t_k}\right)^{1/\nu}\\
		&\lesssim  \sqrt{\nu}\,\exp(c\nu^2 t) t^{q/2}  (\log k)^\beta 
			\qquad\text{a.s.\ on $\bcA_{k,1}$}\\
		&\lesssim\sqrt{\nu}\,\exp(c\nu^2 t) (\log k)^{-\kappa q +\beta} \\
		&\leq \sqrt{\nu}\,\exp\left(\frac{2^{2\kappa} c\nu^2}{8 (\log k)^{2\kappa} }
			\right) (\log k)^{-\kappa q +\beta}
			&\text{[see \eqref{t_k} and \eqref{t:t_k}]},
	\end{align*}
	where the implied constants are deterministic and independent of $\nu\in[2\,,\infty)$
	and $k\in\N$. We now choose $\nu$ slightly more carefully. Namely, we apply the preceding
	with the following particular choice:
	\[
		\nu =(\log k)^{\alpha/2}.
	\]
	Since we have chosen $q \in (0, 1/2)$,  \eqref{alpha} and \eqref{eq:lambda} imply that 
	\[
		\theta_1: = 2\kappa - \alpha>\alpha\left(\frac{2}{q} -1\right) >0 
		\qquad \text{and} \qquad \theta_2:= \kappa q -\beta - \frac{\alpha}{4}
		>\frac{3\alpha}{4}  - \beta  >0.
	\] 
	Therefore,
	\[
		\sqrt{\nu}\,\exp\left(\frac{2^{2\kappa} c\nu^2}{8 (\log k)^{2\kappa} }
		\right) (\log k)^{-\kappa q +\beta} = 
		\exp\left(\frac{2^{2\kappa}c}{8 (\log k)^{\theta_1}}\right) (\log k)^{-\theta_2}.
	\]
	Thus, we find that, for 
	the particular choice of $\nu=\nu(k)$, as given above, we can find a non-random
	number $C=C(\alpha\,,\beta\,,\lip(\sigma))>0$ such that
	\[
		\E\left(\left. \adjustlimits\sup_{s\in(0,t)}
		\sup_{x\in\T}| \mathcal{I}_{v,k+1}(s+t_k\,,x) |^\nu
		\ \,\ \right|\, \F_{t_k}\right)^{1/\nu}\\
		\le C (\log k)^{-\theta_2}
		\qquad\text{a.s.\ on $\bcA_{k,1}$},
	\]
	 simultaneously for all $k\in\N$ large enough to ensure that $\nu=\nu(k)\ge2$.
	 The magnitude of the minimum $k$ that does this is deterministic
	 and depends only on $\alpha$; see  \eqref{interval}. This bound
	 and Chebyshev's inequality together yield that as long as $k$ is large enough
	 to ensure that $\nu\ge2$,
	\[
		\P\left(\left. \sup_{x\in\T}| \mathcal{I}_{v,k+1}(t+t_k\,,x) | > 
		\frac{\e C}{(\log k)^{\theta_2}}\ \,\right|\, \F_{t_k}\right)
		\le \e^{-\nu}\lesssim\exp\left( - (\log k)^{\alpha/2} \right)
		\qquad\text{a.s.\ on $\bcA_{k,1}$},
	\]
	where the implied constant is non random and independent of $k$. By choosing
	a slightly large implied constant, we can see that the preceding holds uniformly for
	all $k\in\N$, in fact. Recall $\chi$ from \eqref{pre:t:t_k}, and recall also that
	$\chi \ge 1.08$. We apply the
	preceding probability bound together with the triangle inequality in order to deduce
	from  \eqref{vkII:gg} and \eqref{pre:t:t_k} that
	\begin{align*}
		&\P\left(\left.  \inf_{x\in B(x_k,2\sqrt t)}| v_k(t+t_k\,,x) | \le  
			\exp(-2\chi)
			\ \,\right|\, \F_{t_k}\right) \\
		&\le \P\left(\left.  \inf_{x\in B(x_k,2\sqrt t)}| v_k(t+t_k\,,x) | \le  
			\e^{-\chi} - \frac{\e C}{(\log k)^{\theta_2}}
			\ \,\right|\, \F_{t_k}\right) \\
		&\le \P\left(\left. \sup_{x\in\T}| \mathcal{I}_{v,k+1}(t+t_k\,,x) | > 
			\frac{\e C}{(\log k)^{\theta_2}}\ \,\right|\, \F_{t_k}\right)	
			\lesssim\exp\left( - (\log k)^{\alpha/2}  \right)
			\quad\text{a.s.\ on $\bcA_{k,1}\cap\bcA_{k,2}\cap\bcA_{k,3}$},
	\end{align*}
	valid for all $k\in\N$ large enough to ensure that
	$\exp(-\chi)-\e C/ (\log k)^{\theta_2} \ge \exp(-2\chi)$, and 
	where the implied constant is non random and independent of $k$. 
	We can increase the implied constant if needed so that it still only depends on
	$(\alpha\,,\lip(\sigma))$ and yet the above inequality holds for all $k\in\N$.
	
	Define 
	\begin{align*}
	\bm{\mathcal{B}}_0 &= \bm{\mathcal{B}}_{0,k} = 
		\left\{ \omega\in\Omega:\,  | v_k(t_k\,,x_k)| (\omega) \ge  \frac{1}{4} \right\} \quad \text{and} \\
		\bm{\mathcal{B}}_j &= \bm{\mathcal{B}}_{j,k} = 
		\left\{ \omega\in\Omega:\,  
		\inf_{x\in B(x_k,j\sqrt t)}| v_k(jt+t_k\,,x) |(\omega) \ge  
		\e^{-2j\chi}\right\}\qquad[j\geq 1].
	\end{align*}
	We have shown that:
	\begin{compactenum}
	\item $\bm{\mathcal{B}}_0\supset\bcA_{k,2}$ a.s. [see \eqref{x_k}]; and
	\item $\P(\bm{\mathcal{B}}^c_1\mid\F_{t_k})\lesssim\exp(-(\log k)^{\alpha/2})$
		a.s.\ on $\bcA_{k,1}\cap\bcA_{k,2}\cap\bcA_{k,3}\cap\bm{\mathcal{B}}_0$ for all $k$.
	\end{compactenum}
	Now we proceed inductively, and repeat the above procedure, in order to see that
	uniformly for every $j\in\Z_+$ and $k\in\N$,
	\[
		\P\left( \bm{\mathcal{B}}_{j+1}^c \mid \F_{t_k+ jt}\right)
		\lesssim\exp\left( -(\log k)^{\alpha/2}\right)\quad\text{a.s.\
		on $\left( \bcA_{k,1}\cap\bcA_{k,2}\cap\bcA_{k,3}\right)\circ\vartheta_{jt}
		\cap\bm{\mathcal{B}}_j$},
	\]
	where $\vartheta$ is the same shift  functional on paths that was defined in \eqref{vartheta}.
	And we can repeat the proof of Lemmas \ref{lem:P(A1)},
	\ref{lem:P(A2)}, and  \ref{lem:P(A3)} in order
	to see that
	\[
		\adjustlimits\max_{i\in\{1,2,3\}}\sup_{j\in\Z_+}
		\P\left( \bcA_{k,i}^c\circ\vartheta_{jt}\right) \lesssim\exp\left( - \log k)^{\eta_3}\right)
		\qquad\text{for all large $k\in\N$},
	\]
	where $\eta_3=\eta_3(\alpha\,, \beta)>1$ is the number  given in Lemma \ref{lem:P(A3)}.
	It might help to also recall  that $\eta_3 \leq \eta_1\wedge \eta_2\wedge (\alpha/2)$;
	see \eqref{eta3}. 
	We now let $\eta:=\eta(\alpha\,, \beta):=\eta_3>1$ and then make repeated appeals to Lemmas
	\ref{lem:cond:prob}, \ref{lem:P(A1)},
	\ref{lem:P(A2)}, and \ref{lem:P(A3)} in order to see
	that for all large  $k\in\N$ and $J = J(k)\in\N$,
	\[
		\P( \bm{\mathcal{B}}^c_J )\lesssim J\exp\left( - (\log k)^{\eta} \right),
	\]
	where the implied constant depends neither on $J$ nor $k$.
	Thanks to \eqref{t:t_k},
	\[
	   J\sqrt t \ge 1 \quad\Leftrightarrow\quad J\ge \sqrt{8}\, (\log t_k)^{\kappa},
	\]
	uniformly for all $k\in\N$. Therefore, we may set $J=3(\log t_k)^{\kappa}$ to ensure that
	$J\sqrt {t}\ge1$, whence
	\[
		B\left(x_k\,, J\sqrt t\right) = \T,
	\]
	For this particular choice of $J = J(k)$, and
	regardless of the value of $x_k$ [which is in $\T$ a.s.\ on $\bcA_{k,1}$], \eqref{t_k} implies that 

	\[
		\P\left\{\inf_{x\in\T} v_k(Jt+t_k\,,x) \le
		\exp(-2J\chi ) \right\} = \P(\bm{\mathcal{B}}_J) 
		\lesssim (\log k)^{\kappa} \exp\left( -(\log k)^{\eta}\right),
	\]
	uniformly for all large $k\in\N$. We used \eqref{t_k} in the last inequality above. Because of the definition of $J$, and thanks to \eqref{t_k} again and \eqref{eq:lambda},
	\[
		Jt = \frac{3}{8}(\log t_k)^{-\kappa} \leq \frac{  3}{8} \, 2^{3\kappa}\, (\log k)^{-\kappa} < (\log k)^{-\alpha}
	\]
 for all  $k$ sufficiently large. In other words, we can see that
	$Jt + t_k < t_{k+1}$ for all $k$ sufficiently large [how large is deterministic
	and depends only on $\alpha$ and $q$]. In addition, since  $\chi \geq 1.08$, we can  have that 
	\[
		\exp(-2J\chi) \geq \exp\left( -6\chi (\log t_k)^{\kappa} \right) \ge \exp\left( -7 (\log t_k)^{\kappa} \right). 
	\]
	Set $s_k = Jt$ and collect things in order to see that 
	\[
		s_k+t_k\le t_{k+1}\quad\text{and}\quad
		\P\left\{ \inf_{x\in\T} v_k(s_k+t_k\,,x) \le \exp\left( -7 (\log t_k)^{\kappa} \right)  \right\}
		\lesssim (\log k)^{\kappa} \exp\left( -(\log  k)^{\eta}\right),
	\]
	uniformly for all sufficiently large $k\in\N$. 
	We now apply the Borel--Cantelli lemma and see that $\kappa > \frac{\alpha}{q}$ for every $q \in (0, 1/2)$ from  \eqref{eq:lambda} to conclude   the proof.
\end{proof}

 Recall that the remaining step of the proof of Theorem \ref{th:inf}
is an assertion that says that $\log\|u(t)\|_{L^1(\T)} - \inf_{x\in\T}\log u(t\,,x) \lesssim (\log t)^{\kappa}$ for every $\kappa >2\alpha$ when $t\gg1$. Because of \eqref{t_k}, Lemma \ref{lem:v_k:LB} essentially verifies
this property, but only along the time sequence $\{s_k+t_k\}_{k\in\N}$. The following
allows for extension to all large times.

\begin{lemma}\label{lem:v_k:LB2}
	With probability one,
	\[
		\adjustlimits\inf_{r\in[s_k+t_k,s_{k+1}+t_{k+1}]}\inf_{x\in\T}
		u(r\,,x) \ge \frac12 \inf_{x\in\T} u(s_k+t_k\,,x) 
		\quad\text{for all but a finite number of $k\in\N$}.
	\]
\end{lemma}

\begin{proof}
	To simplify the exposition define
	\[
		m(t) = \inf_{x\in\T} u(t\,,x)\qquad[t\ge0].
	\]
	Now consider the following space-time random fields,
	\[
		Y_k(r) = \frac{u(r)}{m(s_k+t_k)}\qquad\text{for all $r\ge s_k+t_k$},
	\]
	one for every $k\in\N$.
	Our goal is to prove that almost surely,
	\begin{equation}\label{Y:goal}
		\adjustlimits
		\inf_{s_k+t_k \le r\le s_{k+1}+t_{k+1}}\inf_{x\in\T} Y_k(r\,,x)
		\ge \frac12\qquad\text{for all but a finite number of $k\in\N$}.
	\end{equation}
	
	Let us condition on $\F_{s_k+t_k}$ and apply the Markov property at time
	$s_k+t_k$ to see that $\{Y_k(r)\}_{r\ge s_k+t_k}$ solves \eqref{SHE}
	with $\sigma$ replaced by $\Sigma_k$ where
	\[
		\Sigma_k (w) = \frac{\sigma\left( m(s_k+t_k)w\right)}{m(s_k+t_k)}\qquad[w\in\R].
	\]
	Since  $\lip(\Sigma_k)=\lip(\sigma)$ and $\Sigma_k$ is $\F_{s_k+t_k}$-measurable,
	the analysis of the resulting SPDE is the same as the original analysis of \eqref{SHE}.
	In particular, the fact that 
	$\inf_{x\in\T} Y_k(0\,,x) =1$  and the comparison theorem for SPDEs
	(see Shiga \cite{Shiga}) together imply that 
	$Y_k \ge \mathcal{Y}_k$ where $\mathcal{Y}_k$ solves the same 
	SPDE as $Y_k$ but started identically at one.
	Consequently, \eqref{Y:goal} will follow as soon as we can show that
	\begin{equation}\label{y:goal}
		\adjustlimits
		\inf_{s_k+t_k \le r\le s_{k+1}+t_{k+1}}\inf_{x\in\T} \mathcal{Y}_k(r\,,x)
		\ge \frac12\qquad\text{for all but a finite number of $k\in\N$}.
	\end{equation}
	To be sure, recall that $\mathcal{Y}_k$ solves
	\[
		\mathcal{Y}_k(\tau\,,x) = 1 + \int_{(s_k+t_k,\tau)\times\T}
		p_{\tau-s}(x\,,z) \Sigma_k(\mathcal{Y}_k(s\,,z))\,W(\d s\,\d z).
	\]
	Apply Lemma \ref{lem:tp1} and the Markov property at time
	$s_k+t_k$ in order to find a number $C=C(\lip(\sigma))>0$ such that
	\begin{align*}
		&\E\left( \left. |\mathcal{Y}_k(\tau\,,x) - \mathcal{Y}_k(\tau',x')|^\nu\ \, \right|\, 
			\F_{s_k+t_k}\right)^{1/\nu}\\
		&\lesssim\sqrt{\nu}\,
			\exp\left( C\nu^2\left[ s_{k+1}+t_{k+1} - (s_k+t_k) \right]\right)
			\left\{|\tau-\tau'|^{1/4}+|x-x'|^{1/2}\right\},
	\end{align*}
	for all $\nu\in[2\,,\infty)$, $\tau,\tau'\in[s_k+t_k\,,s_{k+1}+t_{k+1}]$,
	and $x,x'\in\T$, and where the implied constant is deterministic and independent of
	$(\nu\,,\tau\,,\tau',x,x')$ as stated. We may also observe that, for every $k\in\N$,
	\[
		\left[ s_{k+1}+t_{k+1} - (s_k+t_k) \right]
		\le t_{k+2} - t_k = (\log (k+1))^{-\alpha} +(\log k)^{-\alpha} \leq 2 (\log k)^{-\alpha}. 
	\]
	Therefore, since $\alpha>5$ (see \eqref{alpha}),  
	we can choose and fix some $\zeta\in( \frac{1}{\alpha}\,, \frac{1}{4})$ 
	such that $\alpha \zeta>1$, and then apply a suitable form
	of the Kolmogorov continuity theorem in order to see that
	\begin{align*}
		\E\left( \sup_{\tau\in[s_k+t_k,s_{k+1}+t_{k+1}]}
		\sup_{x\in\T} |\mathcal{Y}_k(\tau\,,x) - 1|^\nu  \right)^{1/\nu} \lesssim\sqrt\nu\,
		\exp\left( \frac{2C\nu^2}{(\log k)^\alpha}\right) (\log k)^{-\alpha\zeta}
	\end{align*}
	valid for all $\nu\in[2\,,\infty)$ and $k\in\N$, and where the implied constant
	is deterministic and depends only on $(\zeta\,,\lip(\sigma))$. We apply the preceding with
	 $\nu=(\log k)^{\alpha\zeta}$  in order to deduce from the above that, for this particular
	choice of $\nu$, there exists $A=A(\zeta\,,\lip(\sigma))>0$ such that
	\begin{align*}
		\E\left( \sup_{\tau\in[s_k+t_k,s_{k+1}+t_{k+1}]}
		\sup_{x\in\T} |\mathcal{Y}_k(\tau\,,x) - 1|^\nu  \right)^{1/\nu} \leq A
		 (\log k)^{-\alpha\zeta/2}, 
	\end{align*}
	uniformly for all $k\in\N$. In particular, we learn from Chebyshev's inequality
	that, for this particular choice of $\nu=(\log k)^{\alpha\zeta}$ with $\alpha \zeta>1$,
	\begin{align*}
		\P\left\{ \inf_{s_k+t_k \le r\le s_{k+1}+t_{k+1}}\inf_{x\in\T} \mathcal{Y}_k(r\,,x)
			< \frac12\right\}
			&\le \P\left\{ \sup_{\tau\in[s_k+t_k,s_{k+1}+t_{k+1}]}
			\sup_{x\in\T} |\mathcal{Y}_k(\tau\,,x) - 1| > \frac12\right\}\\
		&\le (2A)^\nu (\log k)^{-\alpha\nu\zeta/2}
			=\exp\left( \nu \left[
			\log(2A) - \frac{\alpha\zeta}{2}\log\log k\right]\right)\\
		&= o\left( \exp\left( - (\log k)^{\alpha\zeta} \right) \right)\qquad[k\to\infty].
	\end{align*} 
	Therefore, the Borel-Cantelli lemma implies \eqref{y:goal} and hence the lemma.
\end{proof}

We are ready to prove the remaining parts of Theorem \ref{th:inf}
and conclude the paper.

\begin{proof}[Proof of Theorem \ref{th:inf}]
	Lemma \ref{lem:P(A2)} and the Borel--Cantelli lemma together imply that
	with probability one,
	\[
		\sup_{r\in[t_k,t_{k+1}]}  \|u(r)\|_{L^1(\T)}
		\le 2 \|u(t_k)\|_{L^1(\T)}
		\quad\text{for all but a finite number of $k\in\N$}.
	\]
	In particular, we almost surely have
	\begin{equation}\label{PD}
		\sup_{r\in[t_k,t_{k+2}]}  \|u(r)\|_{L^1(\T)}
		\le 8 \|u(t_{k-1})\|_{L^1(\T)}
		\quad\text{for all but a finite number of $k\in\N$}.
	\end{equation}
	Lemmas \ref{lem:v_k:LB} and \ref{lem:v_k:LB2}
	then imply that for every fixed $\kappa>2\alpha$,  a.s., the following holds for
	all but a finite number of $k\in\N$:
	\begin{align*}
		\adjustlimits\inf_{r\in[s_k+t_k,s_{k+1}+t_{k+1}]}
			\inf_{x\in\T} u(r\,,x) &\ge \frac12 \inf_{x\in\T}
			u(s_k+t_k\,,x) \\
		&\ge \frac{1}{2}  \exp\left( -7(\log t_k)^{\kappa}\right) \|u( t_{k-1} )\|_{L^1(\T)} \\
		&\ge \frac{1}{16} \exp\left( -7(\log t_k)^{\kappa}\right) \sup_{r\in[t_k,t_{k+2}]}  \|u(r)\|_{L^1(\T)};
	\end{align*} 
	see \eqref{PD} for the last line.
	Since the interval $[s_k+t_k\,,s_{k+1}+t_{k+1}]$ is a subset of
	$[t_k\,,t_{k+2}]$, and because $t_{k+2}/t_k\to 1$ as $k\to\infty$,
	it follows from the above that a.s.,
	\[
		\inf_{x\in\T} u(t\,,x) \ge (16)^{-1} \exp\left( -7(\log t)^{\kappa}\right) \|u(t)\|_{L^1(\T)} 
		\qquad\text{for all $t>0$ outside a certain compact $t$-set.}
	\]
	On one hand, this proves that, with probability one,
	\[
		\inf_{x\in\T}\log u(t\,,x) \ge -\log 16 - 7(\log t)^{\kappa} + \log\|u(t)\|_{L^1(\T)}
		\qquad\text{as $t\to\infty$}.
	\]
	On the other hand, Lemma \ref{lem:P(A1)} and the Borel-Cantelli lemma
	together imply that, almost surely,
	\[
		\log \|u(t)\|_{L^\infty(\T)} \le (\beta+o(1)) \log \log t + \log \|u(t)\|_{L^1(\T)}
		\qquad\text{as $t\to\infty$}. 
	\]
	Combine to find that, with probability one,
	\[
		\text{\rm Osc}_{_{\T}}(\log u(t)) \le 7(\log t)^{\kappa} + O(\log\log t)
		\quad\text{as $t\to\infty$.}
	\]
	Because $\kappa$ is an arbitrary number in $(10\,,\infty)$ by \eqref{eq:kappa-gt-8},
	this implies that
	\[
		\limsup_{t\to\infty}\frac{\log\text{\rm Osc}_{_{\T}}(\log u(t))}{\log\log t}\le 10\quad
		\text{a.s.,}
	\]
	which is an equivalent formulation of the theorem.
\end{proof}



\begin{spacing}{0.1}
\footnotesize
\noindent {\bf Davar Khoshnevisan.} Department of Mathematics, University of Utah,
	Salt Lake City, UT 84112-0090, USA,
	\texttt{davar@math.utah.edu}\\[.2cm]
\noindent {\bf Kunwoo Kim.} Department of  Mathematics, Pohang University of Science and Technology (POSTECH), Pohang, Gyeongbuk, Korea 37673,
	\texttt{kunwoo@postech.ac.kr}\\[.2cm]
\noindent\textbf{Carl Mueller.} Department of Mathematics, University of Rochester,
	Rochester, NY 14627, USA,\\
	\texttt{carl.e.mueller@rochester.edu}
\end{spacing}
\bigskip

\end{document}